\documentclass[12 pt,oneside,reqno]{article}
\usepackage[english]{babel}

\usepackage[letterpaper,top=2cm,bottom=2cm,left=3cm,right=3cm,marginparwidth=1.75cm]{geometry}

\usepackage[dvips]{graphicx}
\usepackage{calc}
\usepackage{color}
\usepackage{amsmath}
\usepackage{amssymb}
\usepackage{amscd}
\usepackage{amsthm}
\usepackage{amsbsy}
\usepackage{delarray}
\usepackage{tikz}
\usepackage{tikz-cd}
\usepackage{enumerate}
\usepackage{enumitem}
\usepackage[T1]{fontenc}
\usepackage{enumerate}
\usepackage{stackengine} 
\usepackage{mathrsfs}
\usepackage{hyperref}
\usepackage[whole]{bxcjkjatype}
\usepackage{titlesec}
\usepackage[nottoc]{tocbibind}
\usepackage{dynkin-diagrams}
\usepackage{etoolbox}
\usepackage{dynkin-diagrams}
\def\row#1/#2!{#1_{\IfStrEq{#2}{}{n}{#2}} & \dynkin{#1}{#2}\\}

\usepackage{lipsum}% http://ctan.org/pkg/lipsum

\definecolor{YK}{rgb}{9,0,0}
\definecolor{YKb}{rgb}{0,0,9}

 \date{\vspace{-4ex}}

\begin{document}

%\selectlanguage{french}

%\xyoption{all}

\setlength{\parindent}{5mm}
\renewcommand{\leq}{\leqslant}
\renewcommand{\geq}{\geqslant}
\newcommand{\N}{\mathbb{N}}
\newcommand{\sph}{\mathbb{S}}
\newcommand{\Z}{\mathbb{Z}}
\newcommand{\R}{\mathbb{R}}
\newcommand{\Q}{\mathbb{Q}}
\newcommand{\C}{\mathbb{C}}
\newcommand{\bK}{\mathbb{K}}

\newcommand{\curlL}{\mathcal{L}}
\newcommand{\A}{\mathcal{A}}
\newcommand{\curlP}{\mathcal{P}}
\newcommand{\F}{\mathbb{F}}
\newcommand{\g}{\mathfrak{g}}
\newcommand{\h}{\mathfrak{h}}
\newcommand{\K}{\mathbb{K}}
\newcommand{\RN}{\mathbb{R}^{2n}}
\newcommand{\ci}{c^{\infty}}
\newcommand{\derive}[2]{\frac{\partial{#1}}{\partial{#2}}}
\renewcommand{\S}{\mathbb{S}}
\renewcommand{\H}{\mathbb{H}}
\newcommand{\eps}{\varepsilon}
\newcommand{\lag}{Lagrangian}
\newcommand{\sub}{submanifold}
\newcommand{\homo}{homogeneous}
\newcommand{\qmor}{quasimorphism}
\newcommand{\enum}{enumerate}
\newcommand{\sa}{symplectically aspherical}
\newcommand{\ovl}{\overline}
\newcommand{\wt}{\widetilde}
\newcommand{\hamd}{Hamiltonian diffeomorphism}
\newcommand{\QH}{quantum cohomology ring}
\newcommand{\thm}{Theorem}
\newcommand{\cor}{Corollary}
\newcommand{\hamil}{Hamiltonian}
\newcommand{\propo}{Proposition}
\newcommand{\conjec}{Conjecture}
\newcommand{\asympt}{asymptotic}
\newcommand{\PR}{pseudo-rotation}
\newcommand{\fuk}{\mathscr{F}}
\newcommand{\specinv}{spectral invariant}
\newcommand{\specnorm}{spectral norm}
\newcommand{\wrt}{with respect to}
\newcommand{\suphv}{superheavy}
\newcommand{\suphvness}{superheaviness}
\newcommand{\symp}{symplectic}
\newcommand\oast{\stackMath\mathbin{\stackinset{c}{0ex}{c}{0ex}{\ast}{\bigcirc}}}
\newcommand{\diffeo}{diffeomorphism}
\newcommand{\quasi}{quasimorphism}
\newcommand{\univham}{\widetilde{\Ham}}
\newcommand{\BK}{Biran--Khanevsky}
\newcommand{\BC}{Biran--Cornea}
\newcommand{\TV}{Tonkonog--Varolgunes}
\newcommand{\varol}{Varolgunes}
\newcommand{\mfd}{manifold}
\newcommand{\smfd}{submanifold}
\newcommand{\EP}{Entov--Polterovich}
\newcommand{\pol}{Polterovich}
\newcommand{\suppot}{superpotential}
\newcommand{\CY}{Calabi--Yau}
\newcommand{\degen}{degenerate}
\newcommand{\coeff}{coefficient}
\newcommand{\acs}{almost complex structure}
\newcommand{\prl}{pearl trajectory}
\newcommand{\prls}{pearl trajectories}
\newcommand{\holo}{holomorphic}
\newcommand{\trans}{transversality}
\newcommand{\pert}{perturbation}
\newcommand{\idem}{idempotent}
\newcommand{\critpt}{critical point}
\newcommand{\tordeg}{toric degeneration}
\newcommand{\nbhd}{neighborhood}
\newcommand{\polar}{polarization}
\newcommand{\symplecto}{symplectomorphism}
\newcommand{\symplectotic}{symplectomorphic}
\newcommand{\FOOO}{Fukaya--Oh--Ohta--Ono}
\newcommand{\quadr}{quadric hypersurface}
\newcommand{\monoconst}{monotonicity constant}
\newcommand{\constr}{construction}
\newcommand{\coord}{coordinate}
\newcommand{\decomp}{decomposition}
\newcommand{\hypsurf}{hypersurface}
\newcommand{\NNU}{Nishinou--Nohara--Ueda}
\newcommand{\transv}{transversality}
\newcommand{\fibr}{fibration}
\newcommand{\atf}{almost toric fibration}
\newcommand{\degtion}{degeneration}
\newcommand{\config}{configuration}
\newcommand{\sing}{singular}
\newcommand{\lagsph}{Lagrangian sphere}
\newcommand{\splgen}{split generation}
\newcommand{\wlg}{without loss of generality}
\newcommand{\AG}{algebraic geometry}
\newcommand{\linindep}{linearly independent}
\newcommand{\RDP}{rational double point}
\newcommand{\grass}{Grassmannian}
\newcommand{\charac}{characteristic}
\newcommand{\HF}{Floer homology}

\theoremstyle{plain}
\newtheorem{theo}{Theorem}%[subsection]
\newtheorem{theox}{Theorem}%[subsection]
\renewcommand{\thetheox}{\Alph{theox}}
\numberwithin{theo}{subsection}
\newtheorem{prop}[theo]{Proposition}
\newtheorem{lemma}[theo]{Lemma}
\newtheorem{definition}[theo]{Definition}
\newtheorem*{notation*}{Notation}
\newtheorem*{notations*}{Notations}
\newtheorem{corol}[theo]{Corollary}
\newtheorem{conj}[theo]{Conjecture}
\newtheorem{guess}[theo]{Guess}
\newtheorem{claim}[theo]{Claim}
\newtheorem{question}[theo]{Question}
\newtheorem{prob}[theo]{Problem}
\numberwithin{equation}{subsection}

\newenvironment{demo}[1][]{\addvspace{8mm} \emph{Proof #1.
    ~~}}{~~~$\Box$\bigskip}

\newlength{\espaceavantspecialthm}
\newlength{\espaceapresspecialthm}
\setlength{\espaceavantspecialthm}{\topsep} \setlength{\espaceapresspecialthm}{\topsep}

\newenvironment{example}[1][]{\refstepcounter{theo} %ref=> ca labelise bien!
\vskip \espaceavantspecialthm \noindent \textsc{Example~\thetheo
#1.} }%
{\vskip \espaceapresspecialthm}

\newenvironment{remark}[1][]{\refstepcounter{theo} %ref=> ca labelise bien!
\vskip \espaceavantspecialthm \noindent \textsc{Remark~\thetheo
#1.} }%
{\vskip \espaceapresspecialthm}

\def\bb#1{\mathbb{#1}} \def\m#1{\mathcal{#1}}

\def\momeg{(M,\omega)}
\def\co{\colon\thinspace}
\def\Homeo{\mathrm{Homeo}}
\def\Diffeo{\mathrm{Diffeo}}
\def\Symp{\mathrm{Symp}}
\def\Sympeo{\mathrm{Sympeo}}
\def\id{\mathrm{id}}
\newcommand{\norm}[1]{||#1||}
\def\Ham{\mathrm{Ham}}
\def\lagham#1{\mathcal{L}^\mathrm{\Ham}({#1})}
\def\Hamtilde{\widetilde{\mathrm{\Ham}}}
\def\cOlag#1{\mathrm{Sympeo}({#1})}
\def\Crit{\mathrm{Crit}}
\def\diag{\mathrm{diag}}
\def\Spec{\mathrm{Spec}}
\def\osc{\mathrm{osc}}
\def\Cal{\mathrm{Cal}}
\def\Ker{\mathrm{Ker}}
\def\Hom{\mathrm{Hom}}
\def\FS{\mathrm{FS}}
\def\tor{\mathrm{tor}}
\def\Int{\mathrm{Int}}
\def\PD{\mathrm{PD}}
\def\Spec{\mathrm{Spec}}
\def\momeg{(M,\omega)}
\def\co{\colon\thinspace}
\def\Homeo{\mathrm{Homeo}}
\def\Hameo{\mathrm{Hameo}}
\def\Diffeo{\mathrm{Diffeo}}
\def\Symp{\mathrm{Symp}}
\def\Sympeo{\mathrm{Sympeo}}
\def\id{\mathrm{id}}
\def\Im{\mathrm{Im}}
\def\Ham{\mathrm{Ham}}
\def\lagham#1{\mathcal{L}^\mathrm{Ham}({#1})}
\def\Hamtilde{\widetilde{\mathrm{Ham}}}
\def\cOlag#1{\mathrm{Sympeo}({#1})}
\def\Crit{\mathrm{Crit}}
\def\dim{\mathrm{dim}}
\def\Spec{\mathrm{Spec}}
\def\spec{\mathrm{spec}}
\def\osc{\mathrm{osc}}
\def\Cal{\mathrm{Cal}}
\def\Fix{\mathrm{Fix}}
\def\det{\mathrm{det}}
\def\Ker{\mathrm{Ker}}
\def\im{\mathrm{Im}}
\def\coker{\mathrm{coker}}
\def\Per{\mathrm{Per}}
\def\rank{\mathrm{rank}}
\def\Span{\mathrm{Span}}
\def\Supp{\mathrm{Supp}}
\def\Hof{\mathrm{Hof}}
\def\grad{\mathrm{grad}}
\def\ind{\mathrm{ind}}
\def\Hor{\mathrm{Hor}}
\def\Vert{\mathrm{Vert}}
\def\Re{\mathrm{Re}}
\def\Chek{\mathrm{Chek}}
\def\Clif{\mathrm{Clif}}

\def\loc{\mathrm{coker}}

\def\om{\omega}

\def\red #1{{ \color{red} #1}}
\newcommand{\ol}[1]{\overline{#1}}
\def\sm#1{C^{\infty}({#1})}

\title{Spectral invariants over the integers}
\author{Yusuke Kawamoto, Egor Shelukhin}

\newcommand{\Addresses}{{% additional braces for segregating \footnotesize
  \bigskip
  \footnotesize

   \textsc{Yusuke Kawamoto, Institute for Mathematical Research (FIM), R\"amistrasse 101, 8092 Z\"urich Switzerland}\par\nopagebreak
  \textit{E-mail address}: \texttt{yusuke.kawamoto@math.ethz.ch} 

\medskip

 \textsc{Egor Shelukhin, Department of Mathematics and Statistics, University of Montreal, C.P. 6128 Succ. Centre-Ville, Montreal, QC, H3C 3J7, Canada}\par\nopagebreak
  \textit{E-mail address}: \texttt{egor.shelukhin@umontreal.ca} }}

\maketitle

\begin{abstract}
Spectral invariants are quantitative measurements in symplectic topology coming from Floer homology theory. We study their dependence on the choice of coefficients in the context of Hamiltonian Floer homology. We discover phenomena in this setting which hold for $\Z$-coefficients and fail for all field coefficients. For example, we prove that the spectral norm, an important metric derived from spectral invariants, is unbounded over $\Z$ for complex projective spaces, while it is uniformly bounded over all fields. This allows us to answer a symplectic version of a question of Hingston, originally asked in the setting of the energy functional on the loop space. We also provide applications to Hamiltonian dynamics and Hofer's geometry. \end{abstract}

\tableofcontents

\section{Introduction}\label{intro}

\subsection{Context}

Floer theory has been one of the main tools of study in {\symp} topology since it was introduced in the 80s. Although it is regarded as the {\symp} analogue of Morse theory, the actual construction is more involved and requires many choices and technical subtleties. One of them is the choice of {\coeff} ring; indeed, even for one of the most basic {\lag} submanifolds like $\R P^n$ in the complex projective space $\C P^n$, one needs to be careful about the choice of the ground ring in order to get meaningful Floer-theoretic information. The coefficient field $\Z/2$ is often chosen in this case. Thus, the choice of the ground ring is not just a minor technical matter, but a crucial point that yields different {\symp} consequences. For recent applications of this point of view to Arnol'd's conjecture we refer to \cite{AB,BX,Rez}.

\textit{Spectral invariants} are important symplectic measurements coming from Floer homology theory which have found a wide range of application in the field: to Hofer's geometry and Hamiltonian dynamics in particular. The theme of this paper is to study how the choice of a ground ring to set up Floer homology impacts spectral invariants. We note that usually fields were used as coefficients in this case. We will, however, mostly look at the ring $\Z$ of coefficients.

\subsection{Spectral norm over $\Z$}\label{gamma over Z}

Spectral invariants are real valued {\symp} invariant assigned to a pair of a {\hamil} and a class in the quantum (co)homology ring via a min-max procedure for the action functional: the {\specinv} for a {\hamil} $H$ and a class $a \in QH(X,\omega;R)$ is denoted by $c_R(H,a)$ where $R$ is the ground ring to define Floer and quantum (co)homologies. One important aspect of the {\specinv} theory is that it gives a non-{\degen} bi-invariant metric of the group of {\hamil} {\diffeo}s originally introduced in \cite{Vit92} and extended in \cite{[Sch00], [Oh05]}. More precisely, for a closed {\symp} {\mfd} $ (M,\omega)$, the group of {\hamil} {\diffeo}s $\Ham (M,\omega)$ has a non-{\degen} bi-invariant metric called the spectral metric $d_{\gamma}$ defined by 
\begin{equation*}
    d_{\gamma_R} (\phi, \psi): = \gamma_R (\phi^{-1} \psi)
\end{equation*}
for any $\phi,\psi \in \Ham (M,\omega)$ where $\gamma_R$ is the \textit{spectral norm} defined as\footnote{Here if $H \in \sm{[0,1]\times M, \R}$ generates the Hamiltonian isotopy $\{ \phi^t \}_{t \in [0,1]},$ then $\ol{H} \in \sm{[0,1]\times M, \R}$ is given by $\ol{H}(t,x) = - H(t, \phi^t(x))$ and generates the inverse isotopy $\{ (\phi^t)^{-1} \}_{t \in [0,1]}.$}
\begin{equation*}
\begin{gathered}
     \gamma_R (\phi):= \inf_{\phi_H=\phi} \gamma_R (H), \\
      \gamma_R (H):= c_R (H,[M] ) +c_R (\overline{H},[M] ) .
 \end{gathered}
\end{equation*}

The properties of the spectral norm $\gamma_R$ have been studied extensively. For example, it is well-known that on $\C P^n$, over a field $\K$, the spectral norm $\gamma_\K$  is bounded from above: \begin{equation}\label{C CPn}    
\sup_{\phi \in \Ham( \C P^n  )} \gamma_\K (\phi ) \leq 1, 
\end{equation} 
where we have equipped the complex projective space $\C P^n$ with the standard Fubini--Study form $\omega_{\FS}$ which is normalized so that $\langle \omega_{\FS}, [\C P^1]\rangle =1$. (See \cite{[EP03]}, and \cite{KS21} for a sharpening.)

Our first main result shows that the behavior of the spectral norm can be drastically different according to the choice of the ground ring to set up Floer theory, namely in contrast to \eqref{C CPn}, the spectral norm over $\Z$-{\coeff}s\footnote{Floer homology can be defined over $\Z$-{\coeff}s at least for semi-positive {\symp} {\mfd}s \cite{[MS04]}. It is expected to hold for any closed {\symp} {\mfd} \cite{BX, Rez}.} is \textit{not} bounded from above.

\begin{theox}\label{main CPn}
For $\C P^n$ with $n>1$, we have 
\begin{equation}\label{Z CP2}
    \sup_{\phi \in \Ham(\C P^n )} \gamma_\Z (\phi ) = + \infty.
\end{equation} 
 Moreover, for $\C P^2$, we have 
\begin{equation}\label{Z14 CP2}
    \sup_{\phi \in \Ham(\C P^2 )} \gamma_{\Z/14} (\phi ) = + \infty .
\end{equation}
    
\end{theox}

\subsection{Application: Question of Hingston}\label{PD CPn}

Hingston, in her research on closed geodesics, studied certain quantities that can be regarded as {\specinv}s for the energy functional on the loop space. Roughly speaking, to a class of the homology of the loop space of a {\mfd} $M$, say $\alpha \in H_\ast (\Lambda M)$, she assigns a real number $c(\alpha)$ via a variational procedure and asked the following question.

\begin{question}[Hingston's question]
    Does there exist a {\mfd} $M$ and a homogeneous non-torsion class $\alpha \in H_\ast (\Lambda M)$ such that 
    \begin{equation*}
        c ( k \cdot \alpha ) < c (\alpha)
    \end{equation*}
      for some $k\in \N$?
\end{question}

We refer the reader to Section \ref{Hingston question} for a more detailed explanation concerning this question. While this question remains open, we consider the following {\symp} analogue of it. 

\begin{question}[Hingston's question: {\symp}]\label{symp ver q of hingston}
    Does there exist a {\symp} {\mfd} $(M,\omega)$ and a {\hamil} $H$ on it such that 
    \begin{equation*}
       \inf_{k \in \Z} c_\Z (H , k\cdot [M]) < c_\Z (H , [M])?
    \end{equation*}
  \end{question}

We consider the fundamental class $[M] \in H_{2n}(M;\Z)$ since it is never torsion for $M$ a closed symplectic manifold and it plays a special role in symplectic topology. However, we could ask if, more generally, there exists a non-torsion quantum cohomology class $a \in QH(M;\Z)$ for a symplectic manifold $(M,\om)$ such that $\inf_{k \in \Z} c_\Z (H , k\cdot a) < c_\Z (H , a ).$

In the setting of Morse theory on a closed manifold, the answer to the analogous question is easily seen to be negative for the fundamental class. Indeed, for a Morse function $f:N \to \R$ on a closed orientable manifold $N,$ $c_{R}(f, [N]) = \max f,$ regardless of the coefficient ring $R$ and similarly $c_{\Z}(f, k[N]) = \max f$ for all $k \neq 0,$ as for every ring $R,$ $H_d(\{f<\max f\};R) = 0$ for $d = \dim N.$ 

In the case of a general non-torsion class $a \in QH_*(M;\Z),$ an easy positive answer to this question would follow for $a = b + c,$ where $b$ is non-torsion, $c$ is torsion and $c_{\Z}(H, c) > c_{\Z}(H, b).$ One can construct such examples for Morse theory of closed orientable manifolds $N$ for classes of certain degrees in $(0, \dim N).$ In fact, it is not hard to make $N$ a symplectic manifold, and hence answer this version of the question, since the filtered Floer theory of a $C^2$-small Morse function coincides with its filtered Floer homology. Similarly, if we do not require $a$ to be degree-homogeneous, then this approach solves Hingston's question for an arbitrary metric on $S^n,$ $n \geq 2,$ by \cite{HR13}. However, if $H_*(M;\Z)$ is torsion-free, like it is in our main case of interest, $M= \C P^n,$ this approach is not accessible.

Let us call the following quantity, the \textit{spectral depth} for a {\hamil} $H$:

\begin{equation}
    \beta_{\text{spec}} (H) := c_\Z (H,[M]) - \inf_{k \in \N} c_\Z (H,k [M]) .
\end{equation}

Inspired by Question \ref{symp ver q of hingston}, we obtain the following estimate for the spectral norm over $\Z$, which leads to a positive answer to Question \ref{symp ver q of hingston}.

\begin{theox}\label{hingston positive gamma estimate}
For every Hamiltonian $H$ on $\C P^n,$ we have
\begin{equation}\label{key inequality}
    \gamma_\Z  (H ) \leq   1 + \beta_{\spec} (H)+ \beta_{\spec} (\overline{H}). 
\end{equation}
 \end{theox}

As a consequence of {\thm}s \ref{main CPn} and \ref{hingston positive gamma estimate}, we obtain {\thm} \ref{hingston positive}, which answers the {\symp} version of Hingston's question, namely Question \ref{symp ver q of hingston}, in the positive for $\C P^n$. In fact, our answer is stronger since it applies to all non-zero quantum cohomology classes.

\begin{theo}\label{hingston positive}
For every non-zero class $a \in QH(\C P^n;\Z)$, there is a {\hamil} $H$ such that 
\begin{equation}
    \inf_{k \in \N} c_\Z (H,k \cdot a ) < c_\Z (H,a )  .
\end{equation}
\end{theo}

\begin{remark}
The only progress for the original version of Hingston's question at the time of writing is a result of Chambers--Liokumovich \cite{CL19} which, in contrast to {\thm} \ref{hingston positive}, provides a partial negative answer. See Remark \ref{remark on Hingston question} for further details.
\end{remark}

\begin{remark}
For a given Hamiltonian $H$ on a closed symplectic manifold $M$ and a non-torsion class $a \in QH(M;\Z),$ for instance $a = [M],$ it would be interesting to investigate the quantity \[c(k) = c_{\Z}(k \cdot a, H)\] as a function of $k.$ For example, clearly if $k | l$ then $c(k) \geq c(l).$ Theorem \ref{Z vs Q} implies that \[\inf_{k \in \Z \setminus \{0\}} c(k)= c_\Q := c_{\Q}(a, H).\] In particular, $c(k) \in [c_{\Q}, c(1)]$ for all non-zero $k \in \Z.$ Moreover, if the infimum is not attained, and $M$ is rational in the sense that $\om(\pi_2(M)) = \rho \Z$ for some $\rho \geq 0,$ then $H$ possesses infinitely many distinct contractible one-periodic orbits. 

Indeed, by the spectrality property of spectral invariants, \eqref{spectrality} below, for all $k,$ $c(k) = \mathcal A_H(\ol{x}_k)$ for certain capped orbits $\ol{x}_k.$ If the infimum is not attained, then there exists a subsequence $c(k_i) \to c_\Q$ and $c(k_i) > c_\Q.$ We may assume that it is strictly decreasing and $c(k_i) - c(k_j) < \rho$ for all $i < j.$ Hence the one-periodic orbits ${x}_{k_i}$ underlying $\ol{x}_{k_i}$ are all distinct.
\end{remark}

\subsection{Poincar\'e duality over $\Z$}\label{PD general} In the previous sections, we highlighted results for the complex projective space $\C P^n$ and in this section, we collect results for general closed monotone {\symp} {\mfd}s.

The uniform bound on the spectral norm for $\K$-coefficients, namely \eqref{C CPn}, is a consequence of the so-called \textit{Poincar\'e duality formula for spectral invariants}: let $a \in QH_{k}(M;\bK)$ be a degree $k$ element of the quantum homology $QH(M;\bK)$ of $M,$ then
\begin{equation}\label{PD formula over field}
    -c_{\K} (\ovl{H}, a) = \inf \{c_{\K}(H, b) :  \Pi(a, b) \neq 0 \},
\end{equation}
where $b$ runs over $QH_{2n-k}(M;\bK),$ for a certain bilinear operation $\Pi( -, -): QH_k(M;\bK) \times QH_{2n-k}(M;\bK) \to \bK.$ (See Section \ref{sec: PD}.) This formula says that the {\specinv}s for $\overline{H}$ are recovered from the {\specinv}s for $H$. {\thm} \ref{main CPn} implies that for $\Z$-coefficients, this is no longer the case: the {\specinv}s for $\overline{H}$ \textit{cannot} be recovered from the {\specinv}s for $H$ and this is the reason why the spectral norm exhibits a highly distinct phenomenon \eqref{Z CP2} with $\Z$-coefficients. We explore the failure of the Poincar\'e duality formula for spectral invariants over $\Z$. To establish an appropriate formula for general closed monotone {\symp} {\mfd}s over $\Z$-{\coeff}s, we introduce the following quantity called the \textit{torsion depth} which, roughly speaking, measures the persistence of the torsion classes in $\{HF^\tau (H;\Z)\}_{\tau \in \R}$.

 \begin{definition}\label{beta tor}
 For each degree $k \in \Z$, consider the following map which is induced by the inclusion map: for $\tau, \tau' \in \R,\ \tau < \tau'$,
 \begin{equation}
     \begin{gathered}
     i_{\tau,\tau'} ^\ast: Ext (HF_{k } ^{\tau '  }(H), \Z)  \to Ext (HF_{k } ^{\tau   }(H), \Z) .
     \end{gathered}
 \end{equation}
 We define the following quantities, which we call the torsion depths:
 
 \begin{equation}
     \begin{aligned}
     \beta_{k, \tor} (H) &:= \beta(Ext (HF_{k } ^{\tau }(H), \Z) ) \\
     & := {\inf \{\kappa : \forall \tau, \coker (i_{\tau} ^\ast )}
      = \coker (i_{\tau , \tau +\kappa } ^\ast ) \}, \\
     \beta_{ \tor} (H) & := \max_{k \in \Z} \beta_{k, \tor} (H).
     \end{aligned}
 \end{equation}
 \end{definition}
 
\begin{remark}
The quantity $\beta_{ \tor} (H) $ is well-defined even though the maximum is taken over an infinite set, as $\beta_{k, \tor} (H)$ is $2 N_M$-periodic where $N_M$ is the minimal Chern number of $(M,\omega).$
\end{remark}

The following is the ``corrected version'' of the Poincar\'e duality formula for spectral invariants over $\Z$.

\begin{theox}\label{corrected formula}
Let $(M,\omega)$ be a closed monotone {\symp} {\mfd}. For a {\hamil} $H$ and a class $a \in QH_{\ast} (M;\Z)$, we have the following:
\begin{equation}
         \inf \{c_{\Z}(H, b) :  \Pi(a, b) \neq 0 \} - \beta_{\tor} (H) \leq -c_{\Z} (\ovl{H}, a) 
         \leq \inf \{c_{\Z}(H, b) :  \Pi(a, b) \neq 0 \}.
\end{equation}
In other words,
\begin{equation}\label{eq: PD}
          0 \leq  c_{\Z} (\ovl{H}, a)+ \inf \{c_{\Z}(H, b) :  \Pi(a, b) \neq 0 \} 
         \leq \beta_{ \tor} (H).
\end{equation}
\end{theox}

{\thm} \ref{corrected formula} tells us that if there are no torsion classes in $\{HF^\tau (H;\Z)\}_{\tau \in \R}$ (i.e. $\beta_{ \tor} (H)=0$, similar to when we work over field {\coeff}s), then we have \eqref{PD formula over field}, thus {\specinv}s for $\overline{H}$ are recovered by {\specinv}s for $H$, but the torsion classes in $\{HF^\tau (H;\Z)\}_{\tau \in \R}$ can contribute to {\specinv}s for $\overline{H}$ and violate the relation \eqref{PD formula over field}.

\begin{remark}
    In \eqref{special}, we have a similar result specific for $\C P^n$.
\end{remark}

\begin{remark}
     Strictly speaking, Inequality \eqref{eq: PD} can be improved to the following by keeping track of the degree. Let $a \in QH_k(M;\Z).$ Then
    \begin{equation}
          0 \leq  c_{\Z} (\ovl{H}, a)+ \inf \{c_{\Z}(H, b) :  \Pi(a, b) \neq 0 \} 
         \leq \beta_{2n-k-1, \tor} (H).
\end{equation}
\end{remark}

So far we have introduced two notions of depth, namely the spectral depth and the torsion depth. The following consequence of \eqref{special} gives a comparison of these two notions for the case of $\C P^n$.

\begin{theo}\label{spectral and torsion depth}
Consider $\C P^n$. We have
\begin{equation}\label{}
      \beta_{\spec} (H) \leq \beta_{\tor} (\overline{H})
\end{equation}
for any {\hamil} $H$. 
\end{theo}

\begin{remark}
We expect \eqref{special} and Theorem \ref{spectral and torsion depth} to extend to a larger class of symplectic manifolds. For instance, they extend verbatim to all symplectically aspherical symplectic manifolds $(M,\om),$ that is $[\om]|_{\pi_2(M)} = 0$ and $c_1(TM)|_{\pi_2(M)}=0.$
\end{remark}

\begin{remark}
{\thm}s \ref{hingston positive gamma estimate} and \ref{spectral and torsion depth} immediately give
\begin{equation}\label{BC inequality}
    \gamma_\Z  (H ) \leq \beta_{\tor} (H) +\beta_{\tor} (\overline{H}) +1
\end{equation}
for any {\hamil} $H$ on $\C P^n$. Biran--Cornea \cite{BC21} obtained a similar-looking inequality to \eqref{BC inequality} in the context of Viterbo's spectral bound conjecture. They proved that the spectral norm $\gamma(L,L')$ for every exact {\lag} $L'$ of the cotangent bundle $T^\ast L$, is controlled by the boundary depth (working over $\Z/2$):
    \begin{equation*}
        \gamma_{\Z/2} (L,L') \leq A \cdot \beta_{\Z/2}(L,L') +B
    \end{equation*}
    for some constants $A,B>0$ independent of $L'$.
\end{remark}

Conversely, we can estimate the torsion depth of a Hamiltonian $H$ in terms of its spectral norm. It is likely that there is also a relation between the torsion depth and the boundary depth $\beta_{\Z}(H)$ of $H$ over $\Z.$

\begin{theo}\label{thm: beta tor and gamma}
Let $(M,\om)$ be a closed monotone symplectic manifold. Then $\beta_{\tor}$ is $\gamma_{\Z}$-Lipschitz continuous in the following sense. For every pair of Hamiltonians $F,G$ \begin{equation}|\beta_{\tor}(F)-\beta_{\tor}(G)| \leq \gamma_{\Z}(G \# \ol{F}).\end{equation} In particular, $\beta_{\tor}$ is Hofer-Lipschitz continuous, and for every Hamiltonian $H$ on $M,$ \[\beta_{\tor}(H) \leq \gamma_{\Z}(H).\]
\end{theo}

Lastly, as we mentioned, even though {\thm} \ref{hingston positive gamma estimate} is only for $\C P^n$, the key for the proof is the following property of general closed monotone {\symp} {\mfd}s which relates {\specinv}s defined over different ground rings.

\begin{theox}\label{Z vs Q}
Let $(M,\omega)$ be a closed monotone {\symp} {\mfd} and $a_\Z \in QH(M;\Z)$ a non-torsion cohomology class.
Set $a_\Q = j(a_\Z)$ for the natural map $j: QH(M;\Z) \to QH(M;\Q).$ Then
\begin{equation}
    \inf_{k \in \N} c(H, k \cdot a_\Z ) = c (H, a_\Q) 
\end{equation}
for every {\hamil} $H$ where the left and right hand sides are {\specinv}s over $\Z$ and $\Q$, respectively. 
\end{theox}

\subsection{Dynamical viewpoint: Hofer geometry and pseudo-rotations}\label{Dynamical viewpoint}
One object of recent interest in Hamiltonian dynamics is the \textit{pseudo-rotation}: see for instance \cite{AS23, [Bra15a],[Bra15b],CGG20, CGG22,[GG18],LS22, Sh20,Sh21, Sh22A}. Pseudo-rotations are {\hamil} {\diffeo}s that have the ``minimal expected number of periodic points''. Although, various dynamical properties of pseudo-rotations were studied, there was no study on the geometry of the set of pseudo-rotations.

%(see Definition \ref{def PR})

\begin{question}\label{PR question}
What does the set 
\begin{equation*}
    PR(M,\omega):=\{\phi \in \Ham(M,\omega) \colon \text{$\phi$ is a pseudo-rotation}\} 
\end{equation*}
look like in $\Ham(M,\omega)$?
\end{question}

We prove the first result in this direction, which states that the set $PR(M,\omega)$ is ``small'' in $\Ham(M,\omega)$.

\begin{theox}\label{PR bound}
Consider $\C P^n$ with $n>1$. We have
\begin{equation}\label{Hofer PR}
    \sup_{\phi \in \Ham(\C P^n)} d_{\Hof} (\phi, PR(\C P^n)) = +\infty .
\end{equation}
Moreover, for any {\PR} $\phi \in \Ham(\C P^n)$, we have
\begin{equation}\label{bound PR AS}
    \gamma_\Z (\phi ) \leq 1 . 
\end{equation}
\end{theox}

\begin{remark}\label{remark on bound PR}
    There are several definitions for {\PR}s, but in {\thm} \ref{PR bound}, we use the version given by Ginzburg--G\"urel \cite{[GG18]}; pseudo-rotations on $\C P^n$ are {\hamil} {\diffeo}s with precisely $n+1$ periodic points.
\end{remark}

\begin{remark}\label{rmk: minimal}
The argument proving Theorem \ref{PR bound} extends verbatim to show that whenever $\phi \in \Ham(\C P^n)$ has exactly $n+1$ {\it fixed} points, then $\gamma_{\Z}(\phi) \leq 1.$ Conversely, it shows that whenever $\gamma_{\Z}(\phi) > 1$ then $\phi$ has at least $n+2$ fixed points $x$ with $HF^{\mathrm{loc}}(\phi,x; \Z) \neq 0,$ at least $n+1$ of which have $\mathrm{rank}_\Z HF^{\mathrm{loc}}(\phi,x; \Z) = 1.$ This means that for a suitable prime $p,$ $N(\phi; \F_p) = \sum \dim_{\F_p} HF^{\mathrm{loc}}(\phi,x; \F_p) \geq n+2.$ Therefore by \cite{Sh22B} $\phi$ has infinitely many periodic points.
\end{remark}

\begin{remark}
It follows from the definition of a pseudo-rotation and the properties of boundary depth that a Hamiltonian diffeomorphism $\phi \in \Ham(M,\om)$ of boundary depth $C=\beta(\phi)$ is at Hofer distance $\geq C$ from $PR(M,\om)$ (see e.g. \cite{AS23}). However, we are not aware of examples of manifolds $(M,\om)$ simultaneously admitting pseudo-rotations and Hamiltonian diffeomorphisms of arbitrarily large boundary depth. Such examples do exist for Morse-Bott pseudo-rotations or for torsion Hamiltonian diffeomorphisms (see \cite{Ush13, AS23}): e.g. $M = \Sigma_g \times S^2$ where $\Sigma_g$ is a closed symplectic surface of genus $g \geq 1.$
\end{remark}

We prove similar results for \textit{{\hamil} torsion}. Recall that a {\hamil} diffeomorphism is $k$-torsion, for integer $k \geq 2$, if $\phi^k = \id$.

\begin{theox}\label{torsion bound}
Consider $\C P^n$. For any {\hamil} $k$-torsion $\phi$, we have 
    \begin{equation*}
       \frac{1}{k} \leq \gamma_{\Z} (\phi) \leq 3.
    \end{equation*}
\end{theox}

Finally, as by-products of our methods, we obtain the following results.
   
    \begin{theo}[]\label{Hofer geom results}
Consider $\C P^n$ with $n>1$.
    
    \begin{enumerate}
        \item There are two linearly independent {\EP} {\qmor}s on the group $\Ham(\C P^n)$. Thus, $ \Ham( \C P^n) $ admits a quasi-isometric embedding of $\R^2$. In particular, the group $ \Ham( \C P^n ) $ is not quasi-isometric to the real line $\mathbb{R}$ {\wrt} the Hofer metric.

        \item There is an unbounded {\qmor} on $\Ham(\C P^n)$ that is both $C^0$ and Hofer-Lipschitz continuous. 
    \end{enumerate}
    \end{theo}

\begin{remark}\label{rmk: KP}
The Kapovich-Polterovich question asked whether $\Ham(M,\omega)$ for $(M,\omega)$ the standard two-sphere $S^2$ is quasi-isometric to $\R.$ This has long been an open question in Hofer geometry. It was recently answered in the negative in \cite{[CGHS],[PS]}. Of course the question makes sense for arbitrary closed symplectic manifolds. For instance, the case of closed symplectic surfaces of genus $g \geq 1$ is folklore and many more examples of a similar kind were produced in \cite{Ush13}. This question has also been answered in the negative for the monotone $S^2 \times S^2$ \cite{[FOOO19]} (see also \cite{[EliPol], Kaw22B}), and the del Pezzo surfaces $\mathbb{D}_3, \mathbb{D}_4 $ \cite{KawA}. Theorem \ref{Hofer geom results}, part 1, answers this question in the negative for complex projective spaces $\C P^n,$ $n>1$. The case $n=2$ is part of an unpublished work of Khanevsky and Zapolsky.
\end{remark}

\begin{remark}
The Entov--{\pol}--Py question asked whether there exists an unbounded $C^0$ and Hofer-Lipschitz continuous {\qmor} on $\Ham(M,\om)$ for  $(M,\omega)$ the standard $S^2.$ It was recently answered positively in \cite{[CGHMSS]}. As in Remark \ref{rmk: KP}, this question makes sense for arbitrary closed symplectic manifolds. It was answered positively for quadrics $Q^n \subset \C P^{n+1},$ $n > 1,$ in \cite{KawA, Kaw22B}. Theorem \ref{Hofer geom results}, part 2, answers this question positively for $\C P^n,$ $n>1$.
\end{remark}

\subsection*{Acknowledgments}

\hspace{10pt} YK: I thank Amanda Hirschi for kindly answering my questions on \cite{CHW}, Soham Chanda for helpful discussions on Remark \ref{chek tori var}. This work was done at \'Ecole Normale Sup\'erieure (ENS Paris), Centre de Recherches Math\'ematiques (CRM) of Universit\'e de Montr\'eal, and the Institute for Mathematical Research (FIM) of ETH Z\"urich where I was a PhD student, a CRM-postdoctoral fellow, and a Hermann-Weyl-Instructor, respectively. I thank all the aforementioned institutions for providing an excellent research environment. 

ES: ES was supported by an NSERC Discovery grant, by the Fonds de recherche du Qu\'ebec - Nature et
technologies, by the Fondation Courtois, and by an Alfred P. Sloan Research Fellowship. This work was partially supported by the National Science Foundation under Grant No. DMS-1928930, while ES was in residence at the Simons Laufer Mathematical
Sciences Institute (previously known as MSRI) Berkeley, California during the Fall 2022 semester. 

 Both authors thank the anonymous referee for a careful reading and helpful comments.

\section{Preliminaries}

\subsection{Conventions and notation}

Although every notation is defined in the paper, we give a list for the convenience of the reader. First,

\begin{itemize}
    \item A ring is always assumed to be commutative and unital.

    \item The complex projective space equipped with the standard Fubini--Study form $\omega_{\FS}$ will be denoted by $\C P^n$ (without the Fubini--Study form). The Fubini--Study form $\omega_{\FS}$ is normalized so that $\langle \omega_{\FS} , [\C P^1]\rangle =1$.
\end{itemize}

For a {\symp} {\mfd} $(M,\omega)$, we denote

\begin{itemize}
    \item $c_{R} (H,a)$ or $c_{} (H,a;R)$: the {\specinv} for {\hamil} Floer homology over a ground ring $R$ for a {\hamil} $H$ and a quantum cohomology class $a \in QH(M;R)$.

\item $c_{R} (H,1)$ or $c (H,1_R)$: the {\specinv} for {\hamil} Floer homology over a ground ring $R$ for a {\hamil} $H$ and the unit of the quantum cohomology ring $1:=1_R=[M] \in QH(M;R)$.

\item $\gamma_{R} (H)$: the spectral norm for {\hamil} Floer homology over a ground ring $R$ for a {\hamil} $H$, i.e. $\gamma_{R} (H)= c_{R} (H,1) +c_{R} (\overline{H},1)$.

\end{itemize}

\subsection{Spectral invariant theory}\label{prelim specinv}
 
It is well-known that on a closed monotone\footnote{Recall that a {\symp} {\mfd} $(M,\omega)$ is monotone if we have $\omega|_{\pi_2(M)}=\kappa \cdot c_1|_{\pi_2(M)}$ for some $\kappa >0$ where $c_1$ is the first Chern class of $TM$.} {\symp} {\mfd} $(M,\omega)$, for a choice of a ground ring $R$ and a non-{\degen} {\hamil} $H \in \sm{[0,1]\times M, \R}$\footnote{A {\hamil} $H$ is non-degenerate if the graph of $\id \times \phi_H$ transversally intersect the diagonal $\Delta$ in $M \times M$.}, which we consider as a time-dependent function via $H_t(x) = H(t,x)$, and a choice of a nice Novikov ring $\Lambda^{\downarrow}$, such as the downward Laurent {\coeff}s $\Lambda_{\text{Lau}} ^{\downarrow} $
$$\Lambda_{\text{Lau};R} ^{\downarrow} :=\{\sum_{k\leq k_0 } b_k t^k : k\in \mathbb{Z},b_k \in R \}  ,   $$
or the downward Novikov {\coeff}s $\Lambda_{\text{Nov}} ^{\downarrow}$ 
$$\Lambda_{\text{Nov};R} ^{\downarrow}:=\{\sum_{j=1} ^{\infty} a_j T^{\lambda_j} :a_j \in R, \lambda _j  \in \mathbb{R},\lim_{j\to \infty} \lambda_j =-\infty \} ,$$
one can construct a filtered Floer homology group $\{HF^\tau(H;R):=HF^\tau(H;\Lambda_R ^{\downarrow}) \}_{\tau \in \R}$. In this paper, we use $\Lambda_R:=\Lambda_{\text{Lau};R} ^{\downarrow} $. For two numbers $\tau<\tau'$, the groups $HF^\tau(H;R)$ and $HF^{\tau'}(H;R)$ are related by a map induced by the inclusion map on the chain level:
$$i_{\tau,\tau'}: HF^\tau(H;R) \longrightarrow HF^{\tau'}(H;R) ,$$
and especially we have
$$i_{\tau}: HF^\tau(H;R) \longrightarrow HF^{}(H;R) ,$$
where $HF^{}(H;R)$ is the Floer homology group. There is a canonical ring isomorphism called the Piunikhin--Salamon--Schwarz (PSS)-map \cite{[PSS96]}, \cite{[MS04]}
$$PSS_{H; R} : QH(M,\omega ;R) \xrightarrow{\sim} HF(H;R) ,$$
where $QH(M,\omega;R)$ denotes the quantum cohomology ring of $(M,\omega)$ with $\Lambda$-{\coeff}s, i.e.
$$ QH(X,\omega;R) := H^\ast (M;R) \otimes_R \Lambda_{R} .$$

The ring structure of $QH(M,\omega;R)$ is given by the quantum product, which is a quantum deformation of the intersection product
$$- \ast - :QH(M,\omega;R) \times QH(M,\omega;R) \to QH(M,\omega;R) . $$

The {\specinv}s, which were introduced by Schwarz \cite{[Sch00]} and developed by Oh \cite{[Oh05]} following the idea of Viterbo \cite{Vit92}, are real numbers $\{c_R (H,a ) \in \R\}$ associated to a pair of a {\hamil} $H$ and a class $a \in QH(M,\omega ;R)$ in the following way:
\begin{equation}\label{eq:spec inv} c_R (H,a ) := \inf \{\tau \in \R : PSS_{H; R } (a) \in \Im (i_{\tau})\} .\end{equation}

\begin{remark}
Although the Floer homology is only defined for a non-{\degen} {\hamil} $H$, the {\specinv}s can be defined for any {\hamil} by using the following so-called Hofer continuity property:
\begin{equation}\label{Hofer continuity property}
    c_R(H,a )-c_R(G,a ) \leq \int_{0} ^1 \left(H_t(x) - G_t(x) \right) dt
\end{equation}
for any non-zero class $a \in QH(M,\omega;R),\ H$ and $G$.
\end{remark}

\textit{Spectrality.} When the {\symp} {\mfd} is rational, that is $\langle \omega ,\pi_2(M) \rangle = \rho\Z$ for $\rho \geq 0$, spectral invariants satisfy the \textit{spectrality} property. Namely, for every {\hamil} $H$ and every class $a \in QH(M,\omega ;R)$, there is a capped orbit $\widetilde{z} \in \widetilde{\mathcal{P}} (H)$\footnote{A capped orbit $\widetilde{z} \in \widetilde{\mathcal{P}} (H)$ is a pair of a periodic orbit $z \in \mathcal{P} (H)$ and a continuous map $u\colon \{z \in \C \colon |z|\leq 1 \}\to M$ such that $u|_{\{z \in \C \colon |z|= 1 \}} = z $.} such that
\begin{equation}\label{spectrality}
    c_R(H,a ) = \mathcal{A}_H (\widetilde{z}) .
\end{equation}
In other words $c_R(H,a) \in \spec(H),$ where $\spec(H) = \mathcal{A}_H(\widetilde{\mathcal{P}} (H))$ is a closed nowhere dense subset of $\R$ called the spectrum of $H.$

\textit{Triangle inequality.}
Spectral invariants satisfy the triangle inequality: for {\hamil}s $H,G$ and $a,b \in QH(M,\omega;R)$, we have
$$ c_R(H,a) + c_R(G, b) \geq c_R(H \# G, a \ast b ) $$
where $H \# G (t,x):=H_t(x) + G_t( \left( \phi^t _{H} \right)^{-1} (x)  )$ and it generates the path $t \mapsto \phi^t _H \circ \phi^t _G$ in $\Ham(M,\omega)$.

\textit{Homotopy invariance.}
Note that {\hamil} {\specinv}s satisfy the homotopy invariance, i.e. if two normalized {\hamil}s $H$ and $G$ generate homotopic {\hamil} paths $t \mapsto \phi_H ^t$ and $t \mapsto \phi_G ^t$ in $\Ham(M,\omega)$, then 
$$c_R(H, -) =c_R(G, -) . $$
Thus, one can define {\specinv}s on $\wt{\Ham} (M,\omega)$:
\begin{equation}
\begin{gathered}
 c_R: \wt{\Ham} (M,\omega) \times QH(M,\omega;R) \rightarrow \R \\
 c_R(\wt{\phi},a):= c_R(H ,a)
\end{gathered}
\end{equation}
where the path $t \mapsto \phi_H ^t$ represents the class of paths $\wt{\phi}$.

\textbf{{\EP} {\qmor}s and (super)heaviness.} Based on {\specinv}s, {\EP} built two theories, namely the theory of (Calabi) {\qmor}s and the theory of (super)heaviness, which we briefly review in this section. 

\textbf{Quasimorphisms.} {\EP} constructed a special map on $\wt{\Ham}(M,\omega)$ called a {\qmor}, under some assumptions on the algebraic structure of the quantum homology ring of $(M,\omega)$. Recall that a {\qmor} $\mu$ on a group $G$ is a map to the real line $\R$ that satisfies the following two properties:
\begin{enumerate}
    \item There exists a constant $C>0$ such that 
    $$|\mu(f \cdot g) -\mu(f)-\mu(g)|<C $$
    for any $f,g \in G$.
    
    \item For any $k \in \Z$ and $f \in G$, we have
    $$\mu(f^k)=k \cdot \mu(f) .$$
\end{enumerate} 

The following is {\EP}'s construction of {\qmor}s on $\wt{\Ham}(M,\omega)$.

\begin{theo}[{\cite{[EP03]}}]\label{EP qmor}
Suppose $QH(M,\omega ;R)$ has a field factor, i.e. 
$$ QH(M,\omega;R) = Q \oplus A $$
where $Q$ is a field and $A$ is some algebra. Decompose the unit $1_{M}$ of $QH(M,\omega ;R)$ {\wrt} this split, i.e. 
$$1_{M}=e + a .$$
Then, the {\asympt} {\specinv} of $\wt{\phi}$ {\wrt} $e$ defines a {\qmor}, i.e.
\begin{equation}
    \begin{gathered}
        \zeta_{e}:\wt{\Ham}(M,\omega) \longrightarrow \R \\
         \zeta_{e} ( \wt{\phi}) := \lim_{k \to +\infty} \frac{c_R (\wt{\phi} ^{ k},e )}{k} =\lim_{k \to +\infty} \frac{c_R (H^{\# k},e )}{k}
    \end{gathered}
\end{equation}
where $H$ is any mean-normalized {\hamil} such that the path $t \mapsto \phi_H ^t $ represents the class $\wt{\phi}$ in $\wt{\Ham}(M,\omega)$.
\end{theo}

\begin{remark}\label{EP and Hofer-Lipschitz conti}
    From \eqref{Hofer continuity property}, it is easy to see that {\EP} {\qmor}s are Hofer-Lipschitz continuous.
\end{remark}

\begin{remark}
By slight abuse of notation, we will also see $\zeta_{e}$ as a function on the set of time-independent {\hamil}s:
\begin{equation}
    \begin{gathered}
        \zeta_{e}:C^{\infty}(M) \longrightarrow \R \\
         \zeta_{e} ( H) := \lim_{k \to +\infty} \frac{c_R (H^{\# k},e )}{k}.
    \end{gathered}
\end{equation}
\end{remark}

\textbf{Superheaviness.} {\EP} introduced a notion of {\symp} rigidity for subsets in $(M,\omega)$ called (super)heaviness.

\begin{definition}[{\cite{[EP09]},\cite{[EP06]}}]\label{def of heavy}
Take an idempotent $e \in QH(M,\omega;R )$ and denote the {\asympt} {\specinv} {\wrt} $e$ by $\zeta_{e}$. A subset $S$ of $(M,\omega)$ is called
\begin{enumerate}
    \item $e$-heavy if for any time-independent {\hamil} $H:M \to \R$, we have
$$  \inf_{x\in S} H(x)  \leq \zeta_{e} ( H) , $$

\item $e$-{\suphv} if for any time-independent {\hamil} $H:M \to \R$, we have
$$ \zeta_{e} ( H) \leq \sup_{x\in S} H(x) . $$

\end{enumerate}
\end{definition} 

The following is an easy corollary of the definition of {\suphvness} which is useful.

\begin{prop}[{\cite{[EP09]}}]\label{suphv constant}
Assume the same condition on $QH(M,\omega ;R)$ as in Theorem \ref{EP qmor}. Let $S$ be a subset of $M $ that is $e$-{\suphv}. For a time-independent {\hamil} $H:M \to \R$ whose restriction to $S$ is constant, i.e. $H|_{S}\equiv r,\ r\in \R$, we have
$$\zeta(H)=r .$$
In particular, two disjoint subsets of $(M,\omega)$ cannot be both $e$-{\suphv}.
\end{prop} 

\begin{proof}
    The first part is an immediate consequence of the definition of (super)heaviness. As for the second part, suppose we have two disjoint sets $A,B$ in $(M,\omega)$ that are both $e$-{\suphv}. Consider a {\hamil} $H$ that is
    $$H|_A=0,\ H|_B=1.$$
    Then, by {\suphvness}, we have
    $$1= \inf_{x\in B} H(x) \leq \zeta_{e}(H) \leq \sup_{x \in A} H(x) =0 ,$$
    which is a contradiction. 
\end{proof}

We end this section by giving a criterion for heaviness, proved by {\FOOO} using the closed-open map
$$\mathcal{CO}^0 : QH (M,\omega;R ) \to HF  (L;R)  $$
where $HF  (L;R)$ is the self-intersection Floer homology ring of a {\lag} $L$ over a ground ring $R$.

\begin{theo}[{\cite[Theorem 1.6]{[FOOO19]}}]\label{CO map heavy}

Assume $HF ( L;R ) \neq 0$. If 
$$\mathcal{CO}^0 (e)\neq 0$$
for an idempotent $e\in QH (M,\omega;R)$, then $L$ is $e$-heavy.
\end{theo}

\begin{remark}\label{hv and suphv}
When $\zeta_e$ is {\homo}, e.g. when $e$ is a unit of a field factor of $QH (M,\omega)$ and $\zeta_e$ is an {\EP} {\qmor}, then heaviness and {\suphvness} are equivalent so Theorem \ref{CO map heavy} will be good enough to obtain the {\suphvness} of $L$.
\end{remark}

\section{Proofs}

\subsection{Proofs for Section \ref{gamma over Z}}\label{Proof of main thm}

The aim of this section is to prove Theorem \ref{main CPn}. We start from some results that will be used in the proof. 

The quantum cohomology ring $QH(\C P^n; R)$ of $\C P^n$ over a ring $R$ has the following ring structure \cite{[MS04]}:
\begin{equation*}
    QH(\C P^n; R)= R [x,t] / (x^{n+1}=t) .
\end{equation*}
Therefore, when we work over a field $\K$, then $ QH(\C P^n; \K)$ is a field, and thus by Theorem \ref{EP qmor},
$$\zeta_{\K} (\tilde{\phi}):=  \lim_{k \to +\infty} \frac{c(\tilde{\phi}^k ,1_\K)}{k} $$
defines a homogeneous quasimorphism on $\wt{\Ham}(\C P^n)$ for any ground field $\K$. In the following, we will consider in particular the cases $\K=\C, \Z/2, \Z/7$.

First of all, we will be focusing on two distinguished {\lag} submanifolds in $\C P^n$, namely $\R P^n$ and the Chekanov-type torus $T^n _{\Chek}$. For the precise definition of the Chekanov-type torus $T^n _{\Chek}$, see Section \ref{good torus}.

\begin{prop}\label{C and Z2 suphv}
We have the following {\suphvness} properties for $\C P^n$:
\begin{itemize}
    \item The Chekanov-type torus $T^n _{\Chek}$ is $\zeta_{\C}$-superheavy. Moreover, for $n=2$, it is also $\zeta_{\Z/7}$-superheavy.

    \item $\R P^n$ is $\zeta_{\Z/2}$-superheavy.
\end{itemize}
\end{prop}

\begin{proof}[Proof of {\propo} \ref{C and Z2 suphv}]
The closed-open maps
\begin{equation*}
    \begin{gathered}
        \mathcal{CO}^0 : QH(\C P^n; \C ) \to HF(T^n _{\Chek};\C),\\
        \mathcal{CO}^0 : QH(\C P^2; {\Z/7} ) \to HF(T^2 _{\text{Ch}};\Z/7),\\
         \mathcal{CO}^0 : QH(\C P^n; {\Z/2}) \to HF(\R P^n;\Z/2),
    \end{gathered}
\end{equation*}
 satisfy
\begin{equation*}
    \begin{gathered}
        \mathcal{CO}^0 (1_\C) = 1_{T^n _{\Chek};\C},\\
          \mathcal{CO}^0 (1_{\Z/7}) = 1_{T^2 _{\text{Ch}};\Z/7},\\
         \mathcal{CO}^0 (1_{\Z/2}) = 1_{\R P^n}.
    \end{gathered}
\end{equation*}

\begin{remark}
    Note that the Floer homology of the Chekanov-type torus is non-zero for $\C$-{\coeff}s, see {\propo} \ref{disjoint torus}. For the case of $n=2$ with $\Z/7$-{\coeff}s, see \cite[Section 5.2.5]{Smi17}.
\end{remark}

Thus, by {\thm} \ref{CO map heavy}, $T^n _{\Chek}$ is $1_\C$-heavy and $\R P^n$ is $1_{\Z/2}$-heavy. Moreover, as $QH(\C P^n; \K)$ is a field, we have that $T^n _{\Chek}$ is $1_\C$-superheavy and $\R P^n$ is $1_{\Z/2}$-superheavy (see Remark \ref{hv and suphv}). Similarly, for $n=2$, $T^2 _{\Chek}$ is also $1_{\Z/7}$-{\suphv}.

\end{proof}

\begin{prop}\label{Z2vsC}
For $\C P^n$, we have
$$ \zeta_{\Z/2} \neq \zeta_{\C} .$$
Moreover, for $n=2$, we also have
$$ \zeta_{\Z/2} \neq \zeta_{\Z/7} . $$
\end{prop}

\begin{proof}[Proof of {\propo} \ref{Z2vsC}]
By {\propo} \ref{disjoint torus} which we will discuss later, $\R P^n$ and $T^n _{\Chek}$ are disjoint, thus by {\propo} \ref{suphv constant} and {\propo} \ref{C and Z2 suphv}, we conclude that $\R P^n$ is not $1_\C$-superheavy (nor $1_{\Z/7}$-superheavy for the $n=2$ case). Thus, we have 
 $$ \zeta_{\Z/2} \neq \zeta_{\C} $$
and, for $n=2$, we also have
$$ \zeta_{\Z/2} \neq \zeta_{\Z/7} . $$
\end{proof}

\begin{remark}
    For $n=2$, we do not know whether or not we have 
    $$\zeta_{\C} = \zeta_{\Z/7} .$$
\end{remark}

\begin{proof}[Proof of Theorem \ref{main CPn}]
    We define a map 
    \begin{equation*}
        \mu: \wt{\Ham} (\C P^n) \to \R 
    \end{equation*}
    as follows:
    \begin{equation}\label{def of mu}
        \mu := \zeta_{\C} -  \zeta_{\Z/2} .
    \end{equation}

The definition \eqref{def of mu} of the function $\mu$ uses {\specinv}s coming from Floer homology over ground fields of different characteristics, and thus we are unable to compare them directly. In order to overcome this issue, we prove the following result.

\begin{prop}\label{key obs}
Let $R$ and $R'$ be rings. Let $j\colon QH(M;R) \to QH(M;R') $ be the map induced by a homomorphism $j: R \xrightarrow[]{} R'.$

Then, we have
\begin{equation*}
    c (H, j(a) ; R' )  \leq   c  ( H , a ; R )
\end{equation*}
for every {\hamil} $H$ and $a \in QH(M;R)$.
\end{prop}

\begin{proof}[Proof of {\propo} \ref{key obs}]
It is enough to prove the case where $j(a)\neq 0$ and a non-{\degen} $H$. By abuse of notation, we also denote by $j\colon HF ^{\tau}(H;R) \to HF ^{\tau}(H;R')$, the morphism that the map $j: R \xrightarrow[]{} R'$ induces on Floer homology. Consider the diagram
   
\begin{equation}\label{diagram coeff}
\begin{tikzcd}
     HF ^{\tau}(H;R) \arrow{d}{j} \arrow{r}{i^\tau _\ast} &  HF_\ast (H;R)  \arrow{d}{j} &   \arrow{l}{PSS_{H;R}  } QH (M; R)  \arrow{d}{j}  \\
     HF ^{\tau}(H;R') \arrow{r}{i^\tau _\ast}  & HF (H;R')  &  \arrow{l}{PSS_{H;R'} } QH (M;R') . 
\end{tikzcd}
\end{equation}
As $j$ preserves the action filtration, the diagram commutes. We can now prove the statement directly by a diagram chase. Namely, if $a' = PSS_{H;R}(a)$ is the image $a' = i^{\tau}_*(a'')$ of some $a'' \in HF ^{\tau}(H;R),$ then $b' = j(a') = PSS_{H,R'}(j(a))$ satisfies $b' = j(a') = j \circ i^{\tau}_*(a'') = i^{\tau}_*(b'')$ for $b'' = j(a'') \in HF ^{\tau}(H;R).$ Hence $c (H, j(a) ; R' )  \leq   c  ( H , a ; R )$ by definition, see \eqref{eq:spec inv}. However, we find the following chain-level approach more instructive.

Fix any class $a \in QH(M;R)$. Take a cycle $\sum_i a_i \cdot \wt{z_i} \in CF(H;R)$ that realizes the {\specinv} $ c  ( H , a ; R )$, i.e.
\begin{equation}
    \begin{gathered}
        c  ( H , a ; R ) = \mathcal{A}_{H} (\sum_i a_i \cdot \wt{z_i}),\\
        [\sum_i a_i \cdot \wt{z_i}] = PSS_{H;R} (a) .
    \end{gathered}
\end{equation}
The cycle $\sum_i j (a_i) \cdot \wt{z_i} \in CF(H;R')$ satisfies 
\begin{equation}
    \begin{gathered}
       \mathcal{A}_{H} (\sum_i j(a_i) \cdot \wt{z_i})  = \mathcal{A}_{H} (\sum_i a_i \cdot \wt{z_i})=  c  ( H , a ; R ),\\
        [\sum_i j(a_i) \cdot \wt{z_i}] = PSS_{H;R} (j(a)) ,
    \end{gathered}
\end{equation}
so we have
\begin{equation*}
    c (H, j(a) ; R' )  \leq   c  ( H , a ; R ).
\end{equation*}
\end{proof}

Consider the natural embedding
\begin{equation}
    \begin{gathered}
        \Z \xrightarrow{} \C,\\
        k \mapsto k
    \end{gathered}
\end{equation}
and the natural projection
\begin{equation}
    \begin{gathered}
         \Z \xrightarrow{} \Z/2,\\
        k \mapsto [k] .
    \end{gathered}
\end{equation}
From {\propo} \ref{key obs}, we get 
\begin{equation}
    \begin{aligned}
        \mu (  \wt{\phi} ) 
        &= \zeta_{\C} (  \wt{\phi} )  -  \zeta_{\Z_2} (  \wt{\phi} ) \\
        & = \zeta_{\C} (\wt{\phi}  ) +   \zeta_{\Z_2} (\wt{\phi}^{-1} ) \\ 
        & \leq \zeta_{\Z} (\wt{\phi}  ) + \zeta_{\Z} (\wt{\phi}^{-1} )   \\
        &= \overline{\gamma}_\Z ( \wt{\phi} )\\
        & \leq  \gamma_\Z ( \wt{\phi} ) .
    \end{aligned}
\end{equation}
 Thus, we obtain 
\begin{equation}\label{mu leq gamma}
     |\mu| \leq \gamma_\Z .
\end{equation}

\begin{claim}\label{EPP question for CPn}
The map $\mu \colon \wt{\Ham}(\C P^n) \to \R$ descends to a map $\mu \colon \Ham (\C P^n) \to \R$ and defines a non-trivial {\homo} {\qmor}. Moreover, it is locally $C^0$-Lipschitz continuous and Hofer-Lipschitz continuous.
\end{claim}

\begin{proof}[Proof of Claim \ref{EPP question for CPn}]
Recall the following two results.

\begin{theo}[{\cite[Theorem 4(1)]{Kaw22B}}]\label{my lemma}
Let $(M,\omega)$ be a monotone symplectic manifold. Fix a ring $R$ (i.e. a commutative ring with unit). For any $\varepsilon>0$, there exists $\delta>0$ such that if $d_{C^0}( \id,\phi_H)<\delta$, then 
$$\gamma_R (H)<\frac{\dim_{\R} (M)}{N_M}\cdot \lambda_0 +\varepsilon$$
where $N_M$ denotes the minimal Chern number and $\lambda_0$ is the positive period of $\langle \omega , \pi_2(M) \rangle$, i.e. $\langle \omega , \pi_2(M) \rangle = \lambda_0 \cdot \Z$.\footnote{Note that in this paper, the Fubini--Study form $\omega_{\FS}$ is normalized so that $\lambda_0 = 1$ for $\C P^n$.}
\end{theo}

\begin{remark}
    Note that $\gamma_R$ denotes the spectral norm defined for Floer homology over the ground ring $R$; see \cite[Remark 5, 25]{Kaw22B}. In particular, {\thm} \ref{my lemma} applies to $\gamma_\Z$.
\end{remark}

\begin{theo}[{\cite{[Sht01]}, \cite[Proposition 1.3]{[EPP12]}}]\label{shtern}
Let $G$ be a topological group and $\mu:G\to \mathbb{R}$ a homogeneous quasimorphism. Then $\mu$ is continuous if and only if it is bounded on a neighborhood of the identity.
\end{theo}
Now, by \eqref{mu leq gamma}, {\thm}s \ref{my lemma} and \ref{shtern}, we conclude that $\mu: \Ham (\C P^n) \to \R$ is a {\homo} {\qmor} that is $C^0$-continuous. It is non-trivial by {\propo} \ref{Z2vsC}. The function $\mu$ is also Hofer-Lipschitz continuous; $\zeta_{\Z_2}$ and $\zeta_{\C}$ are both {\EP} {\qmor}s and thus are Hofer-Lipschitz continuous (see Remark \ref{EP and Hofer-Lipschitz conti}). We have proven Claim \ref{EPP question for CPn}.
\end{proof}

Now, we complete the proof of {\thm} \ref{main CPn}. From the non-triviality in Claim \ref{EPP question for CPn}, there exists $\phi \in \Ham (\C P^n)$ such that
\begin{equation*}
    \mu (\phi ) \neq 0 .
\end{equation*}
By the homogeneity of $\mu$, we have

\begin{equation*}
    \begin{aligned}
        \gamma_{\Z} ( \phi ^k) & \geq  |\mu ( \phi ^k) | =  k \cdot |\mu (\phi ) |, 
    \end{aligned}
\end{equation*}
hence,
\begin{equation*}
    \begin{aligned}
        \lim_{k \to +\infty} \gamma_{\Z} ( \phi ^k) = +\infty .
    \end{aligned}
\end{equation*}

Thus, we have proven that
$$ \sup_{\phi \in \Ham (\C P^n)} \gamma_\Z (\phi) = +\infty . $$

We can also obtain 
$$ \sup_{\phi \in \Ham (\C P^2)} \gamma_{\Z/14} (\phi) = +\infty  $$ in the same manner. Indeed, we define
$\mu' \colon \Ham (\C P^2) \to \R$ by 
\begin{equation*}
    \mu' := \zeta_{\Z/2} - \zeta_{\Z/7} .
\end{equation*}
 By {\propo} \ref{key obs}, we obtain 
 \begin{equation*}
     |\mu'| \leq \gamma_{\Z/14} 
 \end{equation*}
 where we use
 \begin{equation}
    \begin{gathered}
        \zeta_{1_{\Z/2}} \leq \zeta_{1_{\Z/14}},\\
        \zeta_{1_{\Z/7}} \leq \zeta_{1_{\Z/14}}.
    \end{gathered}
\end{equation}
which comes from the natural monomorphisms
 \begin{equation}
    \begin{gathered}
        \Z/14 \hookrightarrow \Z/2,\\
        \Z/14 \hookrightarrow \Z/7.
    \end{gathered}
\end{equation}
Thus, 
\begin{equation}
    \begin{aligned}
        \mu(\wt{\phi})
        &:=  \zeta_{1_{\Z/2}} (\wt{\phi} ) - \zeta_{1_{\Z/7}} (\wt{\phi}) \\
        & \leq \zeta_{1_{\Z/2}} (\wt{\phi} ) + \zeta_{1_{\Z/7}} (\wt{\phi}^{-1}) \\
        & \leq \zeta_{1_{\Z/14}} (\wt{\phi} ) + \zeta_{1_{\Z/14}} (\wt{\phi}^{-1}) \\
        & = \overline{\gamma}_{\Z/14}(\wt{\phi} ) \\
        & \leq \gamma_{\Z/14} (\wt{\phi} ) 
    \end{aligned}
\end{equation}
and it follows that 
$$ \sup_{\phi \in \Ham (\C P^2)} \gamma_{\Z/14} (\phi) = +\infty  $$
by using the non-triviality and homogeneity of $\mu'$. We have proven {\thm} \ref{main CPn}.
\end{proof}

\subsection{Chekanov-type torus in $\C P^n$}\label{good torus}

The aim of this section is to show that there is a monotone {\lag} torus in $\C P^n$ with $n>1$ that has non-zero Floer homology and is disjoint from $\R P^n$. The existence of a {\lag} torus with these properties was used in the proof of {\thm} \ref{main CPn} in Section \ref{Proof of main thm}. Such a torus was known for $n=2,3$ by Oakley--Usher \cite{OU16}, but is new for $n>3$. We appeal to the recent work of Chanda--Hirschi--Wang \cite{CHW}.

\begin{prop}\label{disjoint torus}
In $(\C P^n, \omega_{\FS})$ with $n>1$, there is a monotone {\lag} torus $T^n _{\Chek}$ that has the following two properties:
\begin{itemize}
    \item It has non-zero Floer homology $HF(T^n _{\Chek})\neq 0$.
    \item It is disjoint from $\R P^n$, i.e. $T^n _{\Chek} \cap \R P^n = \emptyset$.
\end{itemize}
\end{prop}

\begin{remark}
     We do not know whether or not, in the $n=3$ case, the monotone {\lag} torus satisfying properties of {\propo} \ref{disjoint torus} from \cite{OU16} is {\hamil} isotopic to the Chekanov-type torus $T^3 _{\Chek}$ that we consider here.
\end{remark}

\begin{remark}\label{chek tori var}
    There are several constructions of exotic {\lag} tori in $\C P^n$ with $n>2$, e.g. \cite{CS10,PT20,Yua22,CHW}. The {\lag} tori from \cite{CS10,PT20,Yua22} all satisfy the first property, namely the non-vanishing of Floer homology, but it seems unlikely that they satisfy the second property in {\propo} \ref{disjoint torus}. In fact, when $n$ is odd, we can show that $\R P^n $ intersects the {\lag} tori from \cite{CS10,PT20,Yua22}\footnote{We thank Sohan Chanda for pointing this out.}; the {\lag} tori from \cite{CS10,PT20,Yua22} are obtained as a mutation of the Clifford torus. Chanda \cite[{\thm} 1.2]{Ch} (which is stated for $\C$-{\coeff}s, but hold also for $\Z/2$-{\coeff}s) shows that shows that up to change of local systems, the intersection Floer homologies between $\R P^n$ and $T^n _{\Clif}$, and between $\R P^n$ and $T^n _{\Chek}$ do not change. According to Alston \cite{Als11}, the intersection Floer homologies between $T^n _{\Clif}$ and, between $\R P^n$ and $T^n _{\Clif}$ over $\Z/2$-{\coeff}s is non-zero when $n$ is odd, thus we have $HF(\R P^n,T^n _{\Chek};\Z/2 ) \neq 0$. This implies that $\R P^n$ and $T^n _{\Chek}$ intersect when $n$ is odd. We expect this to be also true when $n$ is even.
\end{remark}

\begin{remark}
We believe that there are other approaches to construct a monotone {\lag} torus as in {\propo} \ref{disjoint torus}. Indeed, the natural lift of a suitable {\lag} torus in the quadric hypersurface $Q^{n-1}$ to $\C P^n$ using the Biran decomposition \cite{Bir01} for the {\polar} $(\C P^n , Q^{n-1})$ could be another example generalizing \cite{OU16}. By construction, the resulting {\lag} torus in $\C P^n$ is disjoint from $\R P^n$ and its superpotential is computable by \cite{BK13}. See \cite{DTVW} for further work in this direction.
\end{remark}

\begin{proof}[Proof of {\propo} \ref{disjoint torus}]

First, we introduce the torus that we will consider. In fact, we will show that the \textit{lifted Chekanov torus} constructed in \cite{CHW}, denoted by $\overline{T}_{(1,1,2)}$, satisfies the requirements in {\propo} \ref{disjoint torus}, i.e. $T^n _{\Chek} := \overline{T}_{(1,1,2)}$. We review the relevant part from \cite{CHW}. In \cite{CHW}, \textit{lifted Vianna tori} $\overline{T}_{(a,b,c)}$ of Vianna tori $T_{(a,b,c)}$ were constructed in the following way.

First consider the following moment map:
\begin{equation*}
\begin{gathered}
   \mu_{n} : \C P^n \to \R^{n-2} \\
   z=[z_0:\dots :z_n] \mapsto \frac{1}{|z|^2} (|z_2|^2 ,\cdots, |z_n|^2) .
  \end{gathered}
\end{equation*}

Fix a fiber $F_n:= \mu_n ^{-1} \left( \frac{1}{n} (1,\cdots, 1) \right)$. The taking the {\symp} quotient of $F_n$, we get a symplectomorphism $F_n / T^{n-2} \simeq \C P^2$. We denote the induced quotient map by $q$:

\begin{equation}\label{symp quotient}
 q \colon  F_n  \to \C P^2.
\end{equation}

Now, for a Vianna torus $T_{(a,b,c)}$ in $\C P^2$ corresponding to a Markov pair $(a,b,c)$ (c.f. \cite{Via14}), take

\begin{equation*}
\overline{T}_{(a,b,c)}:= q^{-1} (T_{(a,b,c)}) 
\end{equation*}
and see this as a {\lag} in $\C P^n$ and call it the \textit{lifted Vianna torus}. In fact it turns out to be monotone \cite{CHW}.

In order to see that it is disjoint from $\R P^n$, we introduce an alternative description of the map $q \colon  F_n  \to \C P^2$ by using local coordinates:

Take a Darboux chart

\begin{equation*}
\begin{gathered}
   \psi_{n} : \{z \in \C P^n : z_{n} \neq 0 \} \to \C^{n} \\
   z=[z_0:\dots :z_n] \mapsto \frac{|z_n|}{|z|} (\frac{z_0}{z_n} ,\cdots, \frac{z_{n-1}}{z_n} ) .
  \end{gathered}
\end{equation*}
As the fiber $F_n$ is contained in $\{z \in \C P^n : z_{n} \neq 0 \}$, it makes sense to define

\begin{equation}\label{compose to get q}
\begin{gathered}
  F_n \to \C^n \backslash A \to \C^3 \backslash \{0\} \to \C P^2 \\
  z \mapsto \psi_n (z)=(w_1,\cdots, w_n) \mapsto (w_1,w_2,w_3) \mapsto [w_1:w_2:w_3] 
  \end{gathered}
\end{equation}
where $A:=\{z \in \C^n : |w_1|^2 + |w_2|^2 + |w_3|^2 >0 \}$. It is easy to see that the composition of \eqref{compose to get q} coincides with $q$ in \eqref{symp quotient}. From \eqref{compose to get q}, we see that the intersection of $\R P^n$ with $F_n$ is mapped to $\R P^2$ via $q$:
\begin{equation*}
q (\R P^n \cap F_n) = \R P^2 .
\end{equation*}

Thus, as $\overline{T}_{(1,1,2)} \subset F_n$, we have

\begin{equation}\label{eq1}
    \R P^n \cap  \overline{T}_{(1,1,2)}=\R P^n \cap F_n \cap \overline{T}_{(1,1,2)}
\end{equation}

and 

\begin{equation}\label{eq2}
    \begin{aligned}
        q(\R P^n \cap F_n \cap \overline{T}_{(1,1,2)}) &= q(\R P^n \cap F_n ) \cap  q (\overline{T}_{(1,1,2)}) \\
&= \R P^2 \cap T_{(1,1,2)} \\
&= \emptyset.
    \end{aligned}
\end{equation}
In the final line, we have used that $T_{(1,1,2)} $, which is precisely the Chekanov torus in $\C P^2$ \footnote{For this reason, we call $\overline{T}_{(1,1,2)}$ the Chekanov-type torus in $\C P^n$}, is disjoint from $\R P^2$, c.f. \cite{OU16}. By \eqref{eq1} and \eqref{eq2}, we conclude that the lifted Chekanov-type torus $\overline{T}_{(1,1,2)}$ is also disjoint from $\R P^n$, i.e.
\begin{equation*}
    \R P^n \cap  \overline{T}_{(1,1,2)} = \emptyset.
\end{equation*}

Finally, the non-vanishing of Floer homology is a direct consequence of the existence of a critical point of its {\suppot}, see \cite[Example 4.2]{CHW}. This completes the proof of {\propo} \ref{disjoint torus}.
\end{proof}

\subsection{Proofs for Section \ref{PD CPn} and {\thm} \ref{Z vs Q}}

We start with the proof of {\thm} \ref{Z vs Q} and use it to prove {\thm} \ref{hingston positive gamma estimate}.

\begin{proof}[Proof of {\thm} \ref{Z vs Q}] Note that by Hofer-continuity of spectral invariants, we can and will assume that $H$ is non-degerate.
We first prove $ c (H, a_\Q) \leq \inf_{k \in \N} c(H, k \cdot a_\Z )  $. Fix $k\in \N$ and take a Floer cycle $x:=\sum_{j}a_j \overline{z}_j \in CF(H;\Z)$ that represents $PSS_{H;\Z} (k \cdot a_\Z)$ and realizes the {\specinv} $c(H, k \cdot a_\Z)$, i.e. $[x]=PSS_{H;\Z} (k \cdot a_\Z) $ and $\mathcal{A}_H(x)= c(H, k \cdot a_\Z)$. Seeing $x$ as a cycle in $CF(H;\Q)$ instead of $CF(H;\Z)$, we can consider $\frac{1}{k} \cdot x \in CF(H;\Q)$. This cycle $\frac{1}{k}\cdot  x \in CF(H;\Q)$ represents the class $\frac{1}{k} \cdot PSS_{H;\Q} (k \cdot a_\Q) = PSS_{H;\Q} ( a_\Q)$ (c.f. \eqref{diagram coeff}), thus 
\begin{equation*}
    \begin{aligned}
        c (H, a_\Q) & \leq \mathcal{A}_H (\frac{1}{k}\cdot  x ) \\
        & = \mathcal{A}_H (  x ) \\
        & = c(H, k \cdot a_\Z) .
    \end{aligned}
\end{equation*}
This argument applies to any $k \in \N$ so we have proven $ c (H, a_\Q) \leq \inf_{k \in \N} c(H, k \cdot a_\Z )  $. 

We next show $\inf_{k \in \N} c(H, k \cdot a_\Z )  \leq  c (H, a_\Q) $. Take Floer cycles
\begin{equation}
    \begin{gathered}
        x :=\sum_{j}a_j \cdot \overline{z}_j \in CF(H;\Z), \\
        y :=\sum_{j} \frac{p_j}{q_j}  \cdot \overline{w}_j \in CF(H;\Q)
    \end{gathered}
\end{equation}
that represent $PSS_{H;\Z} (a_\Z)$ and $PSS_{H;\Q} (a_\Q)$ and realize the {\specinv}s $c(H, a_\Z)$ and $c(H, a_\Q)$, respectively, where $a_j \in \Z,\ p_j \in \Z,\ q_j \in \Z_{>0}$. Once again, seeing $x$ as a cycle in $CF(H;\Q)$ instead of $CF(H;\Z)$, it represents $a_\Q$, i.e. 
\begin{equation*}
    [x] = PSS_{H;\Q} (a_\Q).
\end{equation*}
As we also have $[y]= PSS_{H;\Q} (a_\Q) $, there is a chain $z \in CF(H;\Q)$ such that 
\begin{equation*}
    x-y = \partial z .
\end{equation*}
Write $\partial z \in CF(H;\Q)$ as
\begin{equation*}
    \partial z= \sum_{j } \frac{m_j}{n_j} \overline{w}' _j 
\end{equation*}
where ${m_j}, \in\Z,\ n_j \in \Z_{>0}$. Now, let $k:= lcm_{i,j} (q_i,n_j)$. Then, 
\begin{equation*}
\begin{aligned}
    x& = y + \partial z \\
    & = \sum_{j} \frac{p_j}{q_j}  \cdot \overline{w}_j + \sum_{j } \frac{m_j}{n_j} \cdot \overline{w}' _j  
\end{aligned}
\end{equation*}
and thus $ k \cdot y $ and $   k \cdot \partial z $ are chains in $CF(H;\Z)$ from our choice of $k$. The chain $ k \cdot y $ represents the class 
\begin{equation*}
    \begin{aligned}
        [k \cdot y] & = [ k \cdot (y  +   \partial z ) ] \\
        & = [k \cdot x] \\
        & = k \cdot PSS_{H;\Z} (a_\Z) \\
        & =PSS_{H;\Z} (k \cdot a_\Z) .
    \end{aligned}
\end{equation*}

Thus,
\begin{equation*}
    \begin{aligned}
         c(H, k \cdot a_\Z)  & \leq \mathcal{A}_H ( k \cdot y  ) \\
        & = \mathcal{A}_H ( y ) \\
        & = c(H, a_\Q) .
    \end{aligned}
\end{equation*}
Note that the last equality is a consequence of our choice of $y$. Thus, we have
$\inf_{k \in \N} c(H, k \cdot a_\Z )  \leq  c (H, a_\Q) $. This completes the proof of {\thm} \ref{Z vs Q}.
\end{proof}

\begin{remark}
The proof of Theorem \ref{Z vs Q} implies, that whenever $H$ is non-degenerate, the infimum $\inf_{k \in \N} c(H, k \cdot a_\Z )$ is attained at some $k \in \N$ for which $c(H, k \cdot a_\Z ) =  c (H, a_\Q).$
\end{remark}

\begin{remark}\label{CPn Z Q C}
 Furthermore, set $a_{\bK} = j(a_{\Z}),$ where $j: QH(M;\Z) \to QH(M;\bK)$ is the natural map for $\bK$ a field. Then
    \begin{equation}
        c (H, a_\Q) =c (H, a_{\mathbb K})
    \end{equation}
    for every {\hamil} $H$, and every field $\mathbb K$ of characteristic zero. In fact, for any field $\K$, we have a natural field extension $j\colon \Q \to \K$, and for any $a \in QH(M;\Q)$, we have $c_\Q (H,a) =c_\K (H, j(a))$; see for example \cite{AS23}.
\end{remark}

Using {\thm} \ref{Z vs Q}, we can prove {\thm} \ref{hingston positive gamma estimate}.

\begin{proof}[Proof of {\thm} \ref{hingston positive gamma estimate}]

We have
\begin{equation*}
    \begin{aligned}
\gamma_\Z (H) & = c(H, 1_\Z ) + c(\ovl{H}, 1_\Z ) \\
& = \left( c(H, 1_\Z ) -c (H, 1_\Q) \right) + \left( c(\ovl{H}, 1_\Z )   - c (\ovl{H}, 1_\Q) \right) +c (H, 1_\Q) +  c (\ovl{H}, 1_\Q)  \\
&=  \left( c(H, 1_\Z ) - \inf_k c(H, k \cdot 1_\Z ) \right) + \left( c(\ovl{H}, 1_\Z )   -  \inf_k c(\ovl{H}, k \cdot 1_\Z ) \right) + \gamma_\Q (H)  \\
&= \beta_{\spec}(H) + \beta_{\spec}(\overline{H}) + \gamma_\Q (H)\\
& \leq 1+\beta_{\spec}(H) + \beta_{\spec}(\overline{H}) 
     \end{aligned}
\end{equation*}
where in the third line, we used {\thm} \ref{Z vs Q}. We have completed the proof of {\thm} \ref{hingston positive gamma estimate}.
\end{proof}

Now, we will see that the positive answer ({\thm} \ref{hingston positive}) to the {\symp} counterpart (Question \ref{symp ver q of hingston}) of a question of Hingston (Question \ref{Hingston question}) follows from {\thm}s \ref{main CPn} and \ref{hingston positive gamma estimate}.

\begin{proof}[Proof of {\thm} \ref{hingston positive}]
    First, we prove {\thm} \ref{hingston positive} for the fundamental class, i.e. $a=[\C P^n]$. As, by {\propo} \ref{disjoint torus}, $\R P^n$ and $T^n _{\Chek}$ are disjoint inside $\C P^n$, we can take an autonomous (i.e. time-independent) {\hamil} $H:\C P^n \to \R $ such that $H|_{\R P^n}=0$ and $H|_{T^n _{\Chek}}=r+1$ where $r>0$: from {\propo} \ref{C and Z2 suphv}, we have
    \begin{equation}\label{example}
        \begin{gathered}
        \zeta_{1_\C} (H) =r +1 ,\\
        \zeta_{1_{\Z/2}} (H) =0.
        \end{gathered}
    \end{equation}
    By \eqref{mu leq gamma}, we have
    \begin{equation*}
        \begin{gathered}
        r +1 \leq \gamma_{\Z}(H) .
        \end{gathered}
    \end{equation*}
    From {\thm} \ref{hingston positive gamma estimate}, it follows that for this $H$, we have either $\beta_{\spec}(H)= c_\Z (H,[M]) - \inf_{k \in \N} c_\Z (H,k [M]) \geq \frac{r}{2}$ or $\beta_{\spec}(\overline{H}) = c_\Z (\ovl{H},[M]) - \inf_{k \in \N} c_\Z (\ovl{H},k [M]) \geq \frac{r}{2} $. This proves {\thm} \ref{hingston positive} for the fundamental class, i.e. $a=[\C P^n]$.

    We next proceed to the general case. Take a class $a = [\C P^l] $ where $0 \leq l \leq n$. The structure of the quantum homology of $\C P^n$ tells you that 
    \begin{equation}\label{CPn quantum}
    \begin{gathered}[c]
        [\C P^{n-1}]^{\ast (n-l)} = [\C P^l],\\
         [\C P^{n-1}]^{\ast (n+1)} = [\C P^n] \cdot T^{-1}. 
    \end{gathered}
    \end{equation}
    Thus, for every {\hamil} $F$, by combining the triangle inequality and \eqref{CPn quantum}, we have
 \begin{equation}\label{other class estimate one}
    \begin{gathered}
        c_\Z ( F, [\C P^{n}] ) \geq  c_\Z (F, [\C P^l ]) \geq  c_\Z ( F, [\C P^{n}] ) -1 .
    \end{gathered}
    \end{equation}
    The same relation as \eqref{CPn quantum} applies to the classes $k\cdot [\C P^{l}]$ where $k \in \N$ and thus, we also have
    \begin{equation}\label{other class estimate two}
    \begin{gathered}
        c_\Z ( F, k\cdot [\C P^{n}] ) \geq  c_\Z (F, k \cdot [\C P^l ]) \geq  c_\Z ( F, k \cdot [\C P^{n}] ) -1 .
    \end{gathered}
    \end{equation}
Now, let $G$ be a {\hamil} that satisfies 
\begin{equation}
    \beta_{\spec} (G) \geq \frac{r}{2},
\end{equation}
which, we have seen earlier. By relations \eqref{other class estimate one} and \eqref{other class estimate two}, we have
\begin{equation}
    \begin{aligned}
         c_\Z (G,[ \C P^l]) - \inf_{k \in \N} c_\Z (G,k \cdot [\C P^l]) 
       & \geq (c_\Z (G,[ \C P^n]) -1) - \inf_{k \in \N} c_\Z (G,k \cdot [\C P^n]) \\
       & = \beta_{\spec} (G) -1 \\
       & \geq \frac{r}{2} -1,
    \end{aligned}
\end{equation}
and thus, by choosing $r>0$ large enough, we have 
\begin{equation}
    \begin{aligned}
         c_\Z (G,[ \C P^l]) - \inf_{k \in \N} c_\Z (G,k \cdot [\C P^l]) 
       >0 .
    \end{aligned}
\end{equation}
We have completed the proof of {\thm} \ref{hingston positive}.
\end{proof}

\subsection{Integer {\coeff}s vs. prime {\coeff}s}\label{Changing and moving the field characteristic}

We discuss some refinements of {\thm}s \ref{hingston positive gamma estimate} and \ref{Z vs Q}. First, we will see that the analogous statement of {\thm} \ref{Z vs Q} for $\Z/p$-{\coeff}s does not hold, and that we can get a more precise information about the minimizer.

\begin{prop}\label{refinement}
Consider $\C P^n$. Let $p$ be a prime. 
\begin{enumerate}
    \item If $c_{\Z/p} (H,1) < c_{\C} (H,1) $ for a {\hamil} $H$, then we have
    \begin{equation}
        c_{\Z/p} (H,1) < \inf_{k \notin p \N} c_\Z (H, k \cdot 1) .
    \end{equation}

    \item If $c_{\Z/p} (H,1) > c_{\C} (H,1) $ for a {\hamil} $H$, then we have
    \begin{equation}
        \inf_{k \in \N} c_\Z (H, k \cdot 1) < \inf_{k \notin p \N} c_\Z (H, k \cdot 1) ,
    \end{equation}
    i.e. the infimum on the left-hand-side is achieved by an integer that is a multiple of $p$.
\end{enumerate}

\end{prop}

\begin{remark}
    The {\hamil} as in \eqref{example} satisfies the assumption of the first statement for $p=2$. Similarly, there are {\hamil}s satisfying the assumption of the second statement for $p=2$.
\end{remark}

\begin{proof}[Proof for {\propo} \ref{refinement}]
By {\propo} \ref{key obs}, we know that 
\begin{equation}\label{fghjk}
     c_{\Z/p} (H,1) \leq \inf_{k \notin p \N} c_\Z (H, k \cdot 1).
\end{equation}
To prove the first statement, we assume that an equality holds in \eqref{fghjk} for every {\hamil}. Then,
\begin{equation*}
    \begin{aligned}
        \inf_{k \notin p \N} c_\Z (H, k \cdot 1) & = c_{\Z/p} (H,1) \\
        & < c_{\C} (H,1) \\
        & =  \inf_{k \in \N} c_\Z (H, k \cdot 1) ,
    \end{aligned}
\end{equation*}
where we used the assumption and Remark \ref{CPn Z Q C} for the second and third lines, respectively. This is a contradiction.

The second statement can be understood as follows, by using \eqref{fghjk}:
\begin{equation*}
    \begin{aligned}
         \inf_{k \in \N} c_\Z (H, k \cdot 1) &=  c_{\C} (H,1) \\
        &< c_{\Z/p} (H,1) \\
        &\leq \inf_{k \notin p \N} c_\Z (H, k \cdot 1).
    \end{aligned}
\end{equation*}
\end{proof}

We now consider {\specinv}s over fields of characteristic $p$ and see if by varying $p$, one can recover {\specinv}s over $\C$ or $\Z$.

\begin{prop}\label{sup Zp vs C}
Consider $\C P^n$. We have
\begin{equation}
   \inf_{\text{$p$:prime}} c_{\Z/p} (H,1) \leq  c_\C (H, 1) \leq \sup_{\text{$p$:prime}} c_{\Z/p} (H,1) \leq  c_\Z (H, 1)
\end{equation}
for every {\hamil} $H$.
\end{prop}

\begin{remark}
    It would be interesting to figure out whether an equality holds for the third inequality in {\propo} \ref{sup Zp vs C}. Note that the first and the second inequalities can be strict (e.g. $p=2$). 
\end{remark}

\begin{proof}[Proof of {\propo} \ref{sup Zp vs C}]
We first prove $\inf_{p} c_{\Z/p} (H,1) \leq  c_\C (H, 1)$. From {\propo} \ref{key obs}, for a fixed prime $p$, we have
\begin{equation*}
  c_{\Z/p} (H,1) \leq  \inf_{k \notin p \cdot \Z} c_\Z (H,k\cdot 1).
\end{equation*}
Thus,
\begin{equation}\label{isvnscsncscncds}
\begin{aligned}
    \inf_p c_{\Z/p} (H,1) & \leq    \inf_p \inf_{k \notin p \cdot \Z} c_\Z (H,k \cdot 1) \\
    & =\inf_{k \in \Z} c_\Z (H,k \cdot 1).
\end{aligned}
\end{equation}
Note that the argument also works for the point class $pt$ instead of the fundamental class $1$.

The inequality $\sup_{p} c_{\Z/p} (H,1) \leq  c_\Z (H, 1)$ is a straightforward consequence of {\propo} \ref{key obs}. We prove $c_\C (H, 1) \leq \sup_{p} c_{\Z/p} (H,1)$. We have
\begin{equation*}
    \begin{aligned}
        \sup_{p} c_{\Z/p} (H,1) & = - \inf_p (-c_{\Z/p} (H,1)) \\
        & = - \inf_p c_{\Z/p} (\overline{H},pt)  \\
        & \geq - \inf_{k \in \Z} c_\Z (\overline{H}, k \cdot pt) \\
        & = \inf_{k \in \Z} c_\Z (H,k \cdot 1) \\
        & = c_\C (H, 1).
    \end{aligned}
\end{equation*}
We used the Poincar\'e duality over fields for the second equality, the inequality \eqref{isvnscsncscncds} with the point class for the inequality, the following {\propo} \ref{inf pt and 1} for the third equality, and Remark \ref{CPn Z Q C} for the fourth equality. We have proven {\propo} \ref{sup Zp vs C}.
\end{proof}

\begin{prop}\label{inf pt and 1}
Consider $\C P^n$. We have 
\begin{equation*}
    \inf_k c_\Z (H, k \cdot pt_\Z) =- \inf_k c_\Z (\overline{H}, k \cdot 1_\Z)  .
\end{equation*}
for any {\hamil} $H$.
\end{prop}

\begin{proof}[Proof of {\propo} \ref{inf pt and 1}]
We have
    \begin{equation*}
\begin{aligned}
     \inf_k c_\Z (H, k \cdot pt_\Z) &= c_\Q (H, pt) \\
     &=- c_\Q (\overline{H}, 1) \\
     & = - \inf_k c_\Z (\overline{H}, k \cdot 1)   
\end{aligned}
\end{equation*}
where in the second equality, we have used the Poincar\'e duality for $\Q$-{\coeff}s.
\end{proof}

\subsection{Proofs for Section \ref{PD general}}\label{sec: PD}

We prove Theorem \ref{corrected formula}.

\begin{proof}[Proof of Theorem \ref{corrected formula}]

First of all, we re-examine the Poincar\'e duality correspondence in Floer theory with filtration and over $\Z$-{\coeff}s. For $\tau \notin \spec(H),$ we introduce the notation
\begin{equation*}
    CF_{\ast } ^{ \geq  \tau} (H):= CF_{\ast } ^{ } (H) /CF_{\ast } ^{ \leq \tau} (H) .
\end{equation*}

Consider the pairing 
\begin{equation}\label{CF pairing}
    \begin{gathered}
\langle - , - \rangle \colon CF_k ^{\leq \tau} (\ovl{H}) \otimes CF_{2n-k} ^{ \geq -\tau} (H) \to \Z \\
\langle z , w \rangle := \tau ((z \ast_{\text{PP}} w) \circ [M]).
    \end{gathered}
\end{equation}
Here $\ast_{\text{PP}}$ denotes the pair-of-pants product, 
\begin{equation}\label{tau map}
    \begin{gathered}
        \tau \colon \Lambda \to \Z\\
        \sum_{j} a_j \cdot s^j \mapsto a_0,
    \end{gathered}
\end{equation}
we dentify $CF(0)$ and $QH(M;\Z) = H_*(M;\Lambda),$ and $\circ$ denotes the usual homological intersection product. This pairing induces an injective chain homomorphism

\begin{equation*}
    PD: CF_k ^{\leq \tau} (\ovl{H}) \to Hom (CF_{2n-k} ^{ \geq -\tau} (H) , \Z) .
\end{equation*}

 This leads to the following exact diagram:
\begin{equation}\label{PD chain diagram}
\begin{tikzcd}[sep=small]
 0 \arrow[r] & CF_{k} ^{\leq \tau} (\ovl{H}) \arrow[d,"PD"] \arrow[r,"i_{\tau}"] & CF_{k}  (\ovl{H}) \arrow[d,"PD"] \arrow[r,"j_{\tau}"] & CF_k ^{\geq \tau} (\ovl{H}) \arrow[d,"PD"] \arrow[r] & 0 \\
 0 \arrow[r] & Hom ( CF_{2n-k} ^{ \geq -\tau} (H) , \Z ) \arrow[r]  &  Hom ( CF_{2n-k}  (H) ,\Z ) \arrow[r] &  Hom ( CF_{2n - k }  ^{\leq  -\tau} (H) , \Z )  \arrow[r]& 0.
\end{tikzcd}
\end{equation}

\begin{remark}
    For the $PD$ on the right, we have used 
    \begin{equation*}
        \begin{aligned}
           Hom ( CF_{2n-k}  (H) ,\Z )/ Hom ( CF_{2n-k} ^{ \geq -\tau} (H) , \Z ) 
           & = Hom ( CF_{2n-k}  (H) / CF_{2n-k} ^{ \geq -\tau} (H) , \Z ) \\
           & \simeq Hom ( CF_{2n - k }  ^{\leq  -\tau} (H) , \Z )  .
        \end{aligned}
    \end{equation*}
\end{remark}

By passing the pairing \eqref{CF pairing} and the diagram \eqref{PD chain diagram} to homology, we obtain
\begin{equation}\label{HF pairing}
    \begin{gathered}
\langle - , - \rangle \colon HF_k ^{\leq \tau} (\ovl{H}) \otimes HF_{2n-k} ^{ \geq -\tau} (H) \to \Z  ,
    \end{gathered}
\end{equation}
and the exact diagram

\begin{equation}\label{PD HF diagram}
\begin{tikzcd}
 HF_{k} ^{\leq \tau} (\ovl{H}) \arrow[d,"PD"] \arrow[r,"i_{\tau}"] & HF_{k}  (\ovl{H}) \arrow[d,"PD"] \arrow[r,"j_{\tau}"] &
 H (CF_k ^{\geq \tau} (\ovl{H}) ) \arrow[d,"PD"]  \\
  HF^{2n - k }  _{ \geq -\tau} (H)  \arrow[r,"j_{-\tau} ^\ast "]  & 
 HF^{2n-k}  (H)  \arrow[r,"i_{-\tau} ^\ast "] &  HF^{2n - k }  _{\leq  -\tau} (H) ,
\end{tikzcd}
\end{equation}

where we use the notation
    \begin{equation*}
        \begin{aligned}
           HF^{\ast }  _{ \geq \tau} (H) &:=  H( Hom ( CF_{\ast} ^{ \geq \tau} (H) , \Z ) ) ,\\
           HF^{\ast}  _{ \leq \tau} (H) &:=H( Hom ( CF_{\ast} ^{ \leq \tau} (H) , \Z ) ) .
        \end{aligned}
    \end{equation*}
   
\begin{remark} 
    Note that in \cite[Exercise 12.4.7]{MS98}, the notation ``$HF^{\ast }  _{ \tau} (H)$'' is used to refer to what we denote by $HF^{\ast }  _{ \geq \tau} (H)$ is this paper.
\end{remark}

In addition to \eqref{HF pairing}, we have the following two pairings which will be important. The first pairing is the canonical pairing of homology and cohomology:
\begin{equation}\label{homology cohomology pairing}
    \begin{gathered}
( - , - ) \colon  HF^{k} _{ \leq \tau} (H)  \otimes HF_k ^{\leq \tau} (H) \to \Z   .
    \end{gathered}
\end{equation}
The second pairing is
\begin{equation}\label{QH pairing}
    \begin{gathered}
\Pi( - , - ) \colon  QH_k (M;\Z)  \otimes QH_{2n-k} (M;\Z) \to \Z\\
\Pi( a , b ) :=  \tau ( \Delta (a \ast b , [M]))
    \end{gathered}
\end{equation}
where $\tau$ is as in \eqref{tau map} and $\Delta $ is as follows; for a class $\sum_{j} a_j \cdot s^j \in QH_\ast(M;\Z)$, we define $\Delta (\sum_{j} a_j \cdot s^j  , [M]):= \sum_{j} (a_j \circ [M]) \cdot s^j  $ where $\circ$ is the intersection product of homology. The relations between the three pairings \eqref{HF pairing}, \eqref{homology cohomology pairing}, and \eqref{QH pairing} are as follows:
for classes $a  \in QH_k (M;\Z)  $ and $b \in QH_{2n-k} (M;\Z)$, we have
\begin{equation}\label{pairing QH vs HF}
    \begin{gathered}
\Pi( a , b )= \langle PSS_H (a) , PSS_{\overline{H}} (b) \rangle
    \end{gathered}
\end{equation}
and for $\tau = \infty$, we also have 
\begin{equation}\label{pairing HF vc hom cohom}
    \begin{gathered}
\langle z , w \rangle =( PD(z) , w ) 
    \end{gathered}
\end{equation}
for every $z \in HF_k (H)$ and $w \in HF_{2n-k} (\overline{H})$ and $PD \colon  HF_{k} (H) \to HF^{2n-k} (\overline{H}) $.

Now, according to diagram \eqref{PD HF diagram}, the {\specinv} of a class $a\in QH_{\ast} (M;\Z)$ is expressed as 
\begin{equation}\label{c(H,a)}
    \begin{aligned}
    c(\ovl{H},a)& = \inf \{\tau : PSS_{\ovl{H}} (a) \in \Im (i_{\tau})\} \\
    &= \inf \{\tau : j_{\tau} \circ PSS_{\ovl{H}} (a) =0 \} \\
    &= \sup \{\tau : j_{\tau} \circ PSS_{\ovl{H}} (a) \neq 0 \} \\
    &= \sup \{\tau : PD \circ j_{\tau} \circ PSS_{\ovl{H}} (a) \neq 0 \in HF^{2n - k }  _{ \leq -\tau} (H)  \} \\
    & = \sup \{\tau : i_{-\tau} ^\ast \circ  PD \circ PSS_{\ovl{H}} (a) \neq 0 \in HF^{2n - k }  _{ \leq -\tau} (H)  \} \\
    &= - \inf \{\tau : i_{\tau} ^\ast \circ  PD \circ PSS_{\ovl{H}} (a) \neq 0 \in HF^{2n - k }  _{ \leq \tau} (H)  \}.
    \end{aligned}
\end{equation}
where $PSS_{\ovl{H}} : QH_{\ast} (M;\Z) \to HF(\ovl{H};\Z) $.

We apply the (cohomology version of the) universal {\coeff} theorem to the chain complex $(Hom ( CF_{2n - k }  ^{\leq \tau} (H) , \Z ) ,\delta)$, which is a cochain complex of $(CF_{2n - k }  ^{ \leq \tau} (H),\partial)$, and obtain the following:
\begin{equation}\label{uct}
\begin{tikzcd}[sep=small]
 0 \arrow[r] & Ext (HF_{2n - k - 1 }  ^{\leq \tau} (H), \Z) \arrow[r,"f_{\tau}"] & HF^{2n - k }  _{\leq  \tau} (H) \arrow[r,"g_{\tau}"] & Hom (HF_{2n - k }  ^{\leq  \tau} (H), \Z) \arrow[r] & 0.
\end{tikzcd}
\end{equation}

To simplify notation, we introduce the following:
\begin{equation}
    \begin{aligned}
    c&:= c(\ovl{H},a),\\
    \alpha &:= PD \circ PSS_{\ovl{H}} (a) ,\\
    T^\tau (H) &:= Ext (HF_{2n - k - 1 }  ^{ \leq \tau} (H), \Z) , \\
    V^\tau (H) &:= HF^{2n - k }  _{ \leq \tau} (H),\\
    F^\tau (H) &:= Hom (HF_{2n - k }  ^{\leq \tau} (H), \Z),\\
    \beta_{\tor}& :=\beta_{\tor} (H)= \beta (Ext (HF_{2n - k - 1 }  ^{ \leq \tau} (H), \Z)).
    \end{aligned}
\end{equation}
With this notation, the universal {\coeff} theorem \eqref{uct} becomes
\begin{equation}
\begin{tikzcd}
 0 \arrow[r] &  T^\tau (H) \arrow[r,"f_{\tau}"] & V^\tau (H) \arrow[r,"g_{\tau}"] & Hom( F^\tau (H);\Z) \arrow[r] & 0
\end{tikzcd}
\end{equation}
and the equation \eqref{c(H,a)} concerning the {\specinv} $c(\ovl{H},a)$ becomes
\begin{equation}\label{spec inv expression}
    \begin{gathered}
    -c=  \inf \{\tau : i_{\tau} ^\ast (\alpha) \neq 0 \in V^\tau (H)  \}.
    \end{gathered}
\end{equation}

We now begin the proof of 
\begin{equation*}
    - c_\Z (\overline{H},a)   =   \inf \{ c_\Z (H, b): a \circ b  \neq 0 \}  .
\end{equation*}
To prove
\begin{equation*}
    - c_\Z (\overline{H},a)  \leq   \inf \{ c_\Z (H, b): \Pi(a , b)  \neq 0 \} 
\end{equation*}
is easy: this follows from the triangle inequality. 

Now, we prove 
\begin{equation}\label{one side}
    - c_\Z (\overline{H},a)  \geq   \inf \{ c_\Z (H, b): \Pi (a , b)  \neq 0 \} - \beta_{\tor} .
\end{equation}
It suffices to prove the following claim in order to prove \eqref{one side}. 

\begin{claim}\label{claim formula}
We have  
$$g_{ -c +\beta_{\tor} } \circ i_{-c +\beta_{\tor}} ^\ast (\alpha) \neq 0 .$$
\end{claim}

Before we prove Claim \ref{claim formula}, we will see how to obtain \eqref{one side}. By Claim \ref{claim formula}, there exists a class $a_0 \in HF_{2n - \deg(a) }  ^{ -c +\beta_{\tor}} (H)$ such that 
\begin{equation}\label{yuvalllll}
    \begin{aligned}
    (  i_{-c +\beta_{\tor}} ^\ast (\alpha)  , a_0 ) = \left( g_{ -c +\beta_{\tor} } \circ i_{-c +\beta_{\tor}} ^\ast (\alpha) \right) (a_0 )  \neq 0.
    \end{aligned}
\end{equation}
The following diagram is a combination of diagrams \eqref{PD HF diagram} and \eqref{uct}, which might summarize the situation:

\begin{equation}
\begin{tikzcd}[sep=tiny]
  HF_{k} ^{\tau} (\ovl{H})  \arrow[r,"i_{\tau}"] & HF_{k}  (\ovl{H}) \arrow[d,"PD"]  & Hom (HF_{2n - k }  ^{ } (H) ;\Z) \arrow[r,"i_{-\tau} ^\ast "] &  Hom (HF_{2n - k }  ^{ \tau} (H) ;\Z)  \\
  &    HF^{2n-k}  (H)  \arrow[r,"i_{-\tau} ^\ast "] \arrow[ur,"g_{}"]  &  HF^{2n - k }  _{ -\tau} (H)  \arrow[ur,"g_{\tau}"]& \\
  & Ext (HF_{2n - k - 1 }  ^{ \tau} (H), \Z) \arrow[ur,"f_{\tau}"]  & &.
\end{tikzcd}
\end{equation}
We obtain

\begin{equation}\label{viviannnnnnn}
    \begin{aligned}
     (  i_{-c +\beta_{\tor}} ^\ast (\alpha) ,a_0)  
     &= (    \alpha, i_{-c +\beta_{\tor}}(a_0) ) \\
    &= \left(     PD (PSS_{\ovl{H}}(a) ), i_{-c +\beta_{\tor}}(a_0) \right)\\
    &= \langle   PSS_{\ovl{H}}(a) , i_{-c +\beta_{\tor}}(a_0) \rangle ,
    \end{aligned}
\end{equation}
where we have used \eqref{pairing HF vc hom cohom} in the last equation. Thus, from \eqref{yuvalllll} and \eqref{viviannnnnnn}, we have
\begin{equation}\label{eq123}
   \langle   PSS_{\ovl{H}}(a) , i_{-c +\beta_{\tor}}(a_0) \rangle  \neq 0 .
\end{equation}
From \eqref{pairing QH vs HF}, we have
\begin{equation}\label{intersection cohomology}
   \Pi ( a, PSS_{\ovl{H}} ^{-1} \circ i_{-c +\beta_{\tor}}(a_0) )   \neq 0.
\end{equation}
Now, the equation \eqref{intersection cohomology} gives us
\begin{equation}\label{eq1234}
    \begin{gathered}
      \inf \{ c(H, b): \Pi (a , b)  \neq 0 \} \leq c(H , PSS_{\ovl{H}} ^{-1}\circ  i_{-c +\beta_{\tor}}(a_0)) ,
    \end{gathered}
\end{equation}
and we also have 
\begin{equation}\label{tauto}
    c(H , PSS_{\ovl{H}} ^{-1}\circ  i_{-c +\beta_{\tor}}(a_0))  \leq -c +\beta_{\tor}
\end{equation}
which follows tautologically from the definition of {\specinv}s. 
Combining \eqref{eq1234} and \eqref{tauto}, we have 
\begin{equation}
    \begin{gathered}
      \inf \{ c(H, b): \Pi(a, b)  \neq 0 \} -\beta_{\tor} \leq -c .
    \end{gathered}
\end{equation}

We have proven \eqref{one side} and thus, it only remains to prove Claim \ref{claim formula} to complete the proof of {\thm} \ref{corrected formula}.

\begin{proof}[Proof of Claim \ref{claim formula}]
We argue by contradiction: assume that
\begin{equation}\label{assumption}
    \begin{gathered}
      g_{ -c +\beta_{\tor} } \circ i_{-c +\beta_{\tor}} ^\ast (\alpha) = 0  .
    \end{gathered}
\end{equation}

We consider the following diagram for $\beta=\beta_{\tor}$:

\begin{equation}
\begin{tikzcd}[]
0\arrow[r] &[-15pt] Ext (HF_{2n - k - 1 } ^{-c +\beta} (H), \Z) \arrow[r,"f_{ -c +\beta }"] \arrow[d,"i_{-c,-c +\beta } ^\ast "] &[5pt] HF^{2n-k} _{-c +\beta} (H)  \arrow[d,"i_{ -c,-c +\beta } ^\ast "] \arrow[r,"g_{ -c +\beta }"] &[5pt]  Hom (HF_{2n - k }  ^{-c +\beta} (H) ;\Z) \arrow[d,"i_{ -c,-c +\beta } ^\ast "] \arrow[r] &[-15pt]  0 \\
0\arrow[r] &[-15pt] Ext (HF_{2n - k - 1 } ^{-c }(H), \Z) \arrow[r,"f_{ -c }"] &[5pt] HF^{2n-k} _{-c} (H) \arrow[r,"g_{ -c }"] &[5pt]  Hom (HF_{2n - k }  ^{ -c} (H) ;\Z) \arrow[r,] &[-15pt]  0
\end{tikzcd}
\end{equation}
The class $i_{-c +\beta_{\tor}} ^\ast (\alpha) $ is in $HF^{2n-k} _{-c +\beta_{\tor}} (H) $. From the assumption \eqref{assumption} and the exactness of the upper sequence, there exists a class $\alpha_{\tor} \in Ext (HF_{2n - k - 1 } ^{-c +\beta_{\tor}} (H), \Z)$ such that
\begin{equation}
    \begin{gathered}
       i_{-c +\beta_{\tor}} ^\ast (\alpha) =  f_{ -c +\beta_{\tor} } (\alpha_{\tor}) .
    \end{gathered}
\end{equation}
By the commutativity of the left block, we have
\begin{equation}
    \begin{gathered}
       i_{-c,-c +\beta_{\tor} } ^\ast \circ f_{ -c +\beta_{\tor} } (\alpha_{\tor}) = f_{ -c  } \circ i_{-c,-c +\beta_{\tor} } ^\ast (\alpha_{\tor}) =0.
    \end{gathered}
\end{equation}
We have used that $i_{-c,-c +\beta_{\tor} } ^\ast (\alpha_{\tor})=0$ by the definition of the boundary depth $\beta_{\tor}$ (see Definition \ref{beta tor}). Now, this implies 
\begin{equation}
    \begin{gathered}
     i_{-c} ^\ast (\alpha) =  i_{-c,-c +\beta_{\tor} } ^\ast \circ i_{-c +\beta_{\tor}} ^\ast (\alpha) =  i_{-c,-c +\beta_{\tor} } ^\ast \circ f_{ -c +\beta_{\tor} } (\alpha_{\tor}) =0 .
    \end{gathered}
\end{equation}
This contradicts the equation \eqref{spec inv expression}, therefore we conclude that
\begin{equation}
    \begin{gathered}
      g_{ -c +\beta_{\tor} } \circ i_{-c +\beta_{\tor}} ^\ast (\alpha) \neq 0  .
    \end{gathered}
\end{equation}
\end{proof}

This completes the proof of Theorem \ref{corrected formula}.

\end{proof}

\begin{remark}
    For $\C P^n$, we have the following version of the Poincar\'e duality formula:
\begin{equation}\label{special}
    -c_\Z (\overline{H},1) = \inf_k c_\Z (H, k \cdot pt_\Z) - \beta_{\spec} (\overline{H}) 
\end{equation}
for any $H$. Note that the right-hand-side includes information about the {\specinv}s for $\overline{H}$ so this is not exactly a Poincar\'e duality formula.

\begin{proof}[Proof of \eqref{special}]
By using {\thm} \ref{Z vs Q} and the Poincar\'e duality formula over $\Q$, we get 
\begin{equation*}
\begin{aligned}
     -c_\Z (\overline{H},1) & = -\left( c_\Z (\overline{H},1) - c_\Q (\overline{H},1) \right) +c_\Q (\overline{H},1)\\
     & = -\left( c_\Z (\overline{H},1) - \inf_k c_\Z (\overline{H},k \cdot 1) \right) - c_\Q (H,pt_\Q) \\
     & = -\left( c_\Z (\overline{H},1) - \inf_k c_\Z (\overline{H},k \cdot 1) \right) - \inf_k c_\Z (H,k \cdot pt_\Z) \\
     & = \inf_k c_\Z (H, k \cdot pt_\Z) - \beta_{\spec} (\overline{H}) .
\end{aligned}
\end{equation*}

\end{proof}
\end{remark}

We prove {\thm} \ref{spectral and torsion depth}.

\begin{proof}[Proof of {\thm} \ref{spectral and torsion depth}]

For any {\hamil} $H$, we have
\begin{equation}\label{eq ryu}
    \begin{aligned}
        c_\Z (H,[M]) -\inf_{k \in \N} c_\Z (H,k [M])
        & \leq -\inf_{j \in \N} c_\Z (\overline{H},j [pt])+ \beta_{\tor} (\overline{H}) -\inf_{k \in \N} c_\Z (H,k [M])\\
        & = \inf_{k \in \N} c_\Z (H,k [M]) -\inf_{k \in \N} c_\Z (H,k [M]) + \beta_{\tor} (\overline{H}) \\
        & = \beta_{\tor} (\overline{H})  .
    \end{aligned}
\end{equation}
In the second line, we have used {\propo} \ref{inf pt and 1}. We have proven {\thm} \ref{spectral and torsion depth}.
\end{proof}

\begin{remark}
   Note that by taking a homogenization ($\ovl{\beta}_{\tor} (H):= \lim_{k \to +\infty} \frac{\beta_{\tor} (H^{\# k})}{k}$) in the inequality in \eqref{BC inequality}, we get
    \begin{equation*}
        \ovl{\gamma}_\Z (H ) \leq   \ovl{\beta}_{\tor} (H)+ \ovl{\beta}_{\tor} (\ovl{H})
    \end{equation*}
    for any $H$. As $\ovl{\gamma}_\Z \neq 0$, it follows that $\ovl{\beta}_{\tor}$ is also not a trivial function, i.e. $\ovl{\beta}_{\tor} \neq 0 .$
 
\end{remark}

\begin{remark}
It is not clear if we have
\begin{equation}
    \beta_{\tor} (\ovl{H})=\beta_{\tor} (H)
\end{equation}
for all $H$. We suspect there to be a counter-example.
\end{remark}

Finally, we prove Theorem \ref{thm: beta tor and gamma}.

\begin{proof}[Proof of Theorem \ref{thm: beta tor and gamma}]

It is sufficient to prove the theorem for $F, G$ non-degenerate. Indeed, then by Hofer continuity $\beta_{\tor}$ extends to arbitrary Hamiltonians and satisfies the desired Lipschitz continuity property.

Fix $\kappa = \gamma_{\Z}(G \# \ol{F})+\eps,$ $\beta = \beta_{\tor}(G)+\eps,$ for $\eps > 0.$ We will prove that $\beta_{\tor}(F) \leq \kappa + \beta.$ Taking $\eps$ to zero, this shows that $\beta_{\tor}(F) - \beta_{\tor}(G) \leq \gamma_{\Z}(G \# \ol{F})$ and exchanging the roles of $F$ and $G$ finishes the proof.

Set $D = G \# \ol{F}.$ For a Hamiltonian $K,$ degree $k \in \Z,$ and $\tau \in \R,$ set $E(\tau,F) = Ext(HF_{k}^{\leq \tau}(K), \Z),$ $E(K) =E(\infty, K) = Ext(HF_{k}(K), \Z)$ and $c(K) = c_{\Z}(K, [M]).$ Then by \cite{KS21, Sh22B} and $Ext(-,\Z)$ being a contravariant functor, there are maps \[C_{F,G}: E(\tau+\kappa, F) \to E(\tau+c(\ol{D}), G), \;\;\;\; C_{G, F}: E(\tau+c(\ol{D}), G) \to E(\tau,F),\]
such that \[C_{G,F} \circ C_{F,G} = i^*_{\tau, \tau+\kappa}: E(\tau+\kappa, F) \to E(\tau, F).\] Moreover $C_{F,G}$ and $C_{G,F}$ commute with the $i^*_{\tau,\tau'}$ and $i^*_{\tau}$ maps on the respective sides.

Now observe that for a Hamiltonian $K,$ $\coker(i^*_{\tau, \tau+\sigma}) = \coker(i^*_{\tau})$ if and only if $\im(i^*_{\tau, \tau+\sigma}) = \im(i^*_{\tau}).$ Of course by the identity \[ i^*_{\tau} = i^*_{\tau,\tau+\sigma} \circ i^*_{\tau+\sigma},\] we have $\im(i^*_{\tau, \tau+\sigma}) \supset \im(i^*_{\tau}).$ We want to show that for $K=F$ the equality holds for all $\tau$ and $\sigma = \beta+\kappa,$ while we know that for $K=G$ it holds for all $\tau$ and $\sigma = \beta.$

Consider now $a \in \im(i^*_{\tau, \tau+\sigma})$ for $\sigma = \beta+\kappa.$ We will prove that there exists $a'' \in E(\infty, F) \cong Ext(QH_k(M;\Z),\Z)$ such that $a = i^*_{\tau}(a'').$ This would finish the proof. 

Write $a = i^*_{\tau, \tau+\beta+\kappa}(a')$ for $a' \in E(\tau+\sigma, F).$ We can rewrite this as \[ a= C_{G,F} \circ C_{F,G} \circ i^*_{\tau+\gamma, \tau+\sigma}(a')=  C_{G,F} \circ i^*_{\tau+c(\ol{D}), \tau+c(\ol{D})+\beta} \circ C_{F,G}(a').\] Set $b' = C_{F,G}(a') \in E(\tau+c(\ol{D})+\beta, G).$ Then by hypothesis $b = i^*_{\tau+c(\ol{D}), \tau+c(\ol{D})+\beta}(b') = i^*_{\tau+c(\ol{D})}(b'')$ for $b'' \in E(G) \cong Ext(QH(M),\Z).$ Therefore \[a =  C_{G,F} \circ i^*_{\tau+c(\ol{D})}(b'').\]

Now $C_{F,G}: E(F) \to E(G)$ is an isomorphism. So $b'' = C_{F,G}(a'')$ for $a'' \in E(F).$ Hence \[  a =  C_{G,F} \circ i^*_{\tau+c(\ol{D})}(b'') = C_{G,F} \circ i^*_{\tau+c(\ol{D})} \circ C_{F,G}(a'') = \] \[ = C_{G,F} \circ C_{F,G} \circ i^*_{\tau+\kappa}(a'') =  i^*_{\tau,\tau+\kappa} \circ i^*_{\tau+\kappa}(a'') = i^*_{\tau}(a'').\]

\end{proof}

\begin{remark}
The only property of $Ext(-,\Z)$ which we used in the proof is it being a contravariant functor from the category of finitely generated abelian groups to itself. Therefore the proof goes through for other such functors. One can prove a similar statement for covariant functors with the definition of the boundary depth suitably modified. \end{remark}

\subsection{Proofs for Section \ref{Dynamical viewpoint}}

The aim of this section is to prove {\thm}s \ref{PR bound} and \ref{torsion bound}.

\begin{proof}[Proof of {\thm} \ref{PR bound}]
Recall that we define pseudo-rotations on $\C P^n$ as {\hamil} {\diffeo}s that have $n+1$ periodic points, following Ginzburg--G\"urel \cite{[GG18]}. We also refer the reader to \cite{[GG18]} for a definition and discussion of the relevant notions of local Floer homology, which we use in this proof. Now, we prove that for every {\PR} $\phi_H$, we have $\gamma_\Z (H) \leq 1$. For a {\PR} $\phi_H$, we have 
    \begin{equation}
        c_\Z (H, [\C P^n]) = \inf_{k \in \N}  c_\Z (H, k \cdot [\C P^n]),
    \end{equation}
because otherwise $\phi_H$ will have more than $n+1$ fixed points as we explain below. Thus, by {\thm} \ref{Z vs Q}, we have 
\begin{equation*}
    \gamma_\Z (H) = \gamma_\Q (H) \leq 1.
\end{equation*}

We now prove that if $\phi$ has precisely $n+1$ fixed points, then every capped one-periodic orbit $\ol{x}$ of $H$ has local Floer homology $HF^{\loc}_*(H, \ol{x}; \Z)$ over $\Z$ supported in a single degree $m=\mu(\ol{x})$ and is free of rank $1$ over $\Z.$ Furthermore $\ol{x} \mapsto m=\mu(\ol{x})$ is a bijection from the set of capped one-periodic orbits to the set $n+2(n+1)\Z$ of integers $m \equiv n \mod 2(n+1).$ Indeed, this statement is true by \cite[Theorem 3.1]{[GG18]} for coefficients in every field $\mathbb{F}$ in place of $\Z.$ In particular, for $\mathbb{F} = \Q,$ we have every capped one-periodic orbit $\ol{x}$ that $\dim_{\Q} HF^{\loc}_*(H, \ol{x}; \Q) = 1$ and it is supported in a unique degree $\mu(\ol{x}).$ Therefore $\rank_\Z HF^{\loc}_*(H, \ol{x}; \Z) =1.$ Now suppose that for a capped one-periodic orbit $\ol{x}$ of $H,$ $HF^{\loc}_*(H, \ol{x}; \Z)$ has torsion (which is necessarily the case if it is not supported in a single degree). Then there exists a prime $p$ such that $\dim_{\F_p} HF^{\loc}_*(H, \ol{x}; \F_p) > 1.$ This is a contradiction to \cite[Theorem 3.1]{[GG18]} for $\mathbb{F}=\F_p,$ which finishes the proof.

Now suppose that $c_{\Z}(H, k\cdot [\C P^n]) < c_\Z (H, [\C P^n])$ for certain $k \in \N$ and $\phi$ has exactly $n+1$ fixed points. Then there exist two capped one-periodic orbits of $H,$ $\ol{x}, \ol{y}$ with distinct action values and non-trivial local Floer homology over $\Z$ in degree $n.$ But then $\mu(\ol{x}) = \mu(\ol{y}) = n$ by the above argument, a contradiction. 
\end{proof}

\begin{proof}[Proof of {\thm} \ref{torsion bound}]

Given a {\hamil} torsion $\phi$, take any generator $H$, i.e. $ \phi=\phi_H$. From the spectrality property for {\specinv}s \eqref{spectrality}, there is a capped orbit $\wt{z}$ such that 
\begin{equation}\label{feyibvwv}
    c_\Z (H,1 ) = \mathcal{A}_H (\widetilde{z}) .
\end{equation}

As $\phi=\phi_H$ is a {\hamil} torsion, we have \[\Spec (H)=\Spec^{ess} (H;\K):=\{c_\K (H, a) \colon a \in QH(M;\K) \}\] for any field $\K$ (c.f. \cite[{\thm} G]{AS23}), thus there exists a class $a \in QH(\C P^n;\K)$ such that 
\begin{equation}\label{invbvrvvrecv}
    c_\K (H, a ) = \mathcal{A}_H (\widetilde{z}) .
\end{equation}
Putting \eqref{feyibvwv} and \eqref{invbvrvvrecv} together, we get 
\begin{equation}\label{csncsoncosncos}
     c_\Z (H,1 ) =  c_\K (H, a ).
\end{equation}

We now look at the mean-index of the capped orbit $\wt{z}$ and obtain a restriction of the degree of $a$. From \cite{AS23,GG12}, we have the following estimate for the mean-index of the action carrier $\wt{z}$ in \eqref{feyibvwv}:
\begin{equation}\label{bcerbceronoier}
    0 \leq \Delta (H, \wt{z}) \leq 2n .
\end{equation}
Similarly from \eqref{invbvrvvrecv}, we have 
\begin{equation}\label{uoneovncvie}
    |a|- 2n \leq \Delta (H, \wt{z}) \leq |a| .
\end{equation}

Putting \eqref{bcerbceronoier} and \eqref{uoneovncvie} together, we get
\begin{equation}\label{ufenooeivevev}
    0 \leq |a| \leq 4n .
\end{equation}

By arguing exactly analogously, we can show that there is a class $b \in QH(\C P^n;\K)$ such that 
\begin{equation}\label{tv elvvndjvnkdlvd}
    c_\Z (\overline{H},1 ) =  c_\K (\overline{H}, b )
\end{equation}
where
\begin{equation}\label{jkndvlnvldnvkdkv}
    0 \leq |b| \leq 4n .
\end{equation}
Now, 
\begin{equation*}
    \begin{aligned}
        \gamma_\Z (H) & =   c_\Z (H,1 ) + c_\Z (\overline{H},1 )  \\
        & = c_\K (H, a ) + c_\K (\overline{H}, b ) \\
         & = c_\K (H, a ) -  c_\K (H, b' ),
    \end{aligned}
\end{equation*}

where $b' \in QH(\C P^n;\K)$ is the generator of degree $|b'|=2n - |b|$. Thus, we have
\begin{equation}\label{vjshbcsbkcbjsckns}
    -2n \leq |b'| \leq 2n .
\end{equation}

Thus, 

\begin{equation*}
    \begin{aligned}
        \gamma_\Z (H) & = c_\K (H, a ) -  c_\K (H, b' )\\
        & \leq \nu (ab'^{-1}  ) \\
        & \leq 3 .
    \end{aligned}
\end{equation*}
where we used $-2n \leq |ab'^{-1}| \leq 6n $.

The proof for the lower bound in {\thm} \ref{torsion bound} works exactly as \cite[{\thm} J]{AS23}, as the argument only uses spectrality of {\specinv}s and thus it carries over to the $\Z$-{\coeff} case. We have completed the proof of {\thm} \ref{torsion bound}.
 \end{proof}

We expect {\thm} \ref{PR bound} to hold in greater generality and propose the following conjecture.

\begin{conj}\label{conj PR}
Let $(M, \omega)$ be a monotone symplectic manifold admitting a pseudo-rotation. Any {\PR} $\phi \in \Ham (M , \omega)$ satisfies 
\begin{equation}\label{PR and spec norm}
     \sup_{k \in \Z } \gamma_\Z (\phi^k) < +\infty.
\end{equation}
\end{conj}

\begin{remark}
    The converse of {\conjec} \ref{conj PR} does not hold; there are {\hamil} {\diffeo}s satisfying \eqref{PR and spec norm} that are not {\PR}s (e.g. for a {\hamil} $H$ supported in a displaceable open set, $\phi_H$ satisfies \eqref{PR and spec norm}).
\end{remark}

We conclude the section with the proof of {\thm} \ref{Hofer geom results}.

\begin{proof}[Proof of {\thm} \ref{Hofer geom results}]

From {\propo} \ref{Z2vsC}, we have two distinct {\qmor}s $\zeta_{\Z/2}$ and $\zeta_\C$ on $\wt{\Ham} (\C P^n)$. By \cite[Section 4.3]{[EP03]}, these two {\qmor}s descend to $\Ham (\C P^n)$, thus they define two distinct {\qmor}s on $\Ham (\C P^n)$ which are both Hofer-Lipschitz continuous. This completes the first part of the claim. The second part was shown in Claim \ref{EPP question for CPn}.

\end{proof}

\subsection{Remark on $\Z$-{\suphvness}}

From {\propo} \ref{key obs}, it follows that
\begin{equation*}
    \zeta_{\K}(H) \leq \zeta_{\Z}(H):= \lim_{k \to +\infty} \frac{c_\Z (k \cdot H, [M])}{k}
\end{equation*}
for any autonomous {\hamil} $H$, and thus $\zeta_{\Z}$-{\suphvness} is a stronger condition than $\zeta_{\K}$-{\suphvness} for any field $\K$. We point out the following new hierarchy of {\symp} rigidity in terms of {\suphvness}.

\begin{prop}\label{Z suphv}
    Consider $\C P^n$. The Clifford torus $T^n _\mathrm{Clif}$ is $\zeta_{\Z}$-{\suphv} while the Chekanov torus $T^n _\mathrm{Chek}$ and $\R P^n$ are not.
    
    \end{prop}

\begin{proof}[Proof of {\propo} \ref{Z suphv}]
    In \cite[{\thm} 1.11]{[EP09]}, it was proven that the \textit{special fiber} of a monotone {\symp} toric {\mfd} is {\suphv} {\wrt} any idempotent of the quantum homology. Although, \cite[{\thm} 1.11]{[EP09]} was formulated for quantum homology over $\C$, the proof is based on the index of {\hamil} orbits and in particular does not use the Poincar\'e duality and thus it carries over to $\Z$-{\coeff}s. The special fiber in $\C P^n$ is precisely the Clifford torus and thus, we conclude that the Clifford torus is $\zeta_{\Z}$-{\suphv}. As the Chekanov torus $T^n _\mathrm{Chek}$ and $\R P^n$ are disjoint (see {\propo} \ref{disjoint torus}), we conclude that they are not $\zeta_{\Z}$-{\suphv} (even though they are $\zeta_{\C}$-{\suphv} and $\zeta_{\F_2}$-{\suphv}, respectively, see {\propo} \ref{C and Z2 suphv}).
\end{proof}

\section{Discussions}\label{Discussions}

\subsection{Quantum linking form}

In $M= \C P^n$, there might be $ b \in Tor (HF^\tau (H))$, which eventually disappears, such that $ c (\ovl{H}, [M]) $ corresponds to $-\tau$ where $PSS_H ([pt]) + b$ appears at $HF^\tau (H)$. It would be interesting to study the ``filtered quantum linking form'' which is the quantum counterpart of the linking form in singular homology theory (an analogue of the Poincar\'e pairing for torsion classes) by taking the filtration into account. The non-filtered version of this looks as follows:

$$\curlL : QH_{\ast} (M ; \Z) \otimes QH_{2n-1- \ast} (M ; \Q / \Z) \to \Q / \Z$$
which comes from the following operation:
$$ \oast : QH_{i} (M ; \Z) \otimes Tor (QH_{j} (M ; \Z) ) \to QH_{i+j -2n+1} (M ; \Q / \Z) $$
$$(a,b) \mapsto a \ast \beta^{-1} (b)$$
where 
$$\beta : QH_{j} (M ; \Q / \Z) \to QH_{j-1} (M ;  \Z)    $$
denotes the Bockstein homomorphism.

We have for $a \in  QH_{i} (M ; \Z)$, $ b \in  QH_{2n-1-i} (M ; \Z)$,
$$ \curlL (a ,b ) = \langle a \oast b , 1 \rangle  ,\ 1 \in QH^0 (M ; \Q / \Z).$$

\subsection{The Chiang {\lag}}

It was shown by Evans--Likili \cite{EL15} that the Chiang {\lag} $L_{\text{Chi}}$, an exotic {\lag} constructed by Chiang \cite{Chi04}, in $\C P^3$ has non-zero Floer homology group over a field of {\charac} $5$. This implies that 
\begin{equation}\label{CP3 char 5}
    \zeta_{1_{\F_5}} = \overline{\ell}_{L_{\text{Chi}};\F_5},
\end{equation}
where $\overline{\ell}_{L_{\text{Chi}};\F_5}$ denotes the {\asympt} {\lag} {\specinv} for the unit of $QH(L_{\text{Chi}};\F_5)$; although in Section \ref{prelim specinv}, we only defined {\specinv}s for {\hamil} Floer homology, we can define analogous invariants for {\lag} Floer theory (see \cite{LZ18}). Given a {\lag} $L$, we denote the {\lag} {\specinv} for a {\hamil} $H$ and a quantum homology class $a \in QH(L)$ by $\ell(H,a)$.\footnote{When we take the fundamental class $a = [L]$, we abbreviate as $\ell_{L} =\ell (- , [L]) $.}

Konstantinov \cite{Kon} studied Floer homology with higher rank local systems and proved that the {\HF} of the Chiang {\lag} $L_{\text{Chi}}$ in $\C P^3$ is non-zero over a field of {\charac} $2$ with some higher rank local systems. Provided that we have the theory of {\specinv}s for Floer homology with higher rank local systems, we will have 
\begin{equation}\label{CP3 char 2}
    \zeta_{1_{\F_2}} = \overline{\ell}_{(L_{\text{Chi}},\rho);\F_2}.
\end{equation}

We also know that 
\begin{equation*}
    \zeta_{1_{\F_2}} = \overline{\ell}_{\R P^3;\F_2},
\end{equation*}
as $HF(\R P^3;\F_2) \neq 0$. Thus, 
\begin{equation*}
    \overline{\ell}_{(L_{\text{Chi}},\rho);\F_2} = \overline{\ell}_{\R P^3;\F_2},
\end{equation*} 
which is consistent with the fact that $\R P^3$ and $L_{\text{Chi}}$ are not {\hamil} displaceable from one another. It would be interesting to understand whether one can distinguish $\zeta_{1_{\F_2}}$ and $\zeta_{1_{\F_5}}$. Given \eqref{CP3 char 2} and \eqref{CP3 char 5}, they ``support'' the same {\lag} so naively one might guess that they coincide, but it is not clear to the authors. 

It would also be nice to know whether or not the Chiang {\lag} is disjoint from a suitable Chekanov-type torus.

\subsection{Question of Hingston}\label{geodesics}

In \cite{HR13}, Hingston and Rademacher studied a certain {\specinv} in the context of loop homology and considered a similar phenomenon to the theme of this paper. We briefly review the concerning discussion of \cite{HR13}.

Let $M$ be a compact manifold equipped with a Riemannian (or Finsler) metric $g$. Let $R$ be a finitely generated abelian group, e.g. $\Z$, $\Z/k\Z$ with $k \in \N$. Now, denote the homology group of the loop space of $M$ {\wrt} the ground ring $R$ by $H_\ast (\Lambda M; R)$. Let $\Lambda^\tau M$ denote the subspace of free loop space $\Lambda M$ which consists of loops whose square-root of energy\footnote{One could also use the length functional, but following Goresky and Hingston \cite{GH09} we use $F$ for better differentiability properties.} $F(\gamma) = \sqrt{\int_{S^1} ||\dot{\gamma}(t)||dt}$ is at most $\tau.$  For a non-trivial homology class $\alpha \in H_\ast (\Lambda M; R)$, the critical (minimax) level of $\alpha$ (i.e. the {\specinv} for $\alpha$) is defined as follows.

\begin{equation*}
    \begin{aligned} 
        cr (\alpha) &= \inf \{ \tau  : \alpha \in \Im \, H_\ast (\Lambda^\tau M; R) \to H_\ast(\Lambda M; R)  \}
    \end{aligned}
\end{equation*}

\begin{remark}
One can understand this definition intuitively via a ``minimax definition'', as follows: \[cr(\alpha) = \inf_{[x] = \alpha} \sup_{\gamma \in \Im (x)}  F(\gamma ),\] where the singular chain $x$ ranges over all representatives of the homology class $\alpha$, and $\gamma$ ranges over all the points in the image of $x$, which are loops in $M$.
\end{remark}
Hingston poses the following question for these {\specinv}s.

\begin{question}[{\cite[Remark 3.4]{HR13}, \cite[Question 3.2.1]{BM21}}]\label{Hingston question}
    Does there exist a metric on $M=S^n$ and a non-torsion homogeneous homology class $\alpha \in H_d (\Lambda M; Z)$ with
    \begin{equation}
        0 < cr(m\cdot \alpha) < cr(\alpha )
    \end{equation}
for some $m \in  \N$?
\end{question}

\begin{remark}\label{remark on Hingston question}
    Chambers--Liokumovich showed that for $S^2$, $m$ odd and $d=1$, the answer to Question \ref{Hingston question} is actually negative \cite[{\cor} 1.3]{CL19}. See \cite[Section 3]{BM21} for more about this question.
\end{remark}

An analogous question in the context of {\hamil} Floer homology would be the following (Question \ref{symp ver q of hingston}).

\begin{question}[Hingston's question: {\symp}]\label{Hingston question detailed}
    Does there exist a {\symp} {\mfd} $(M,\omega)$ and a {\hamil} $H$ on it such that 
    \begin{equation*}
       \inf_{k \in \Z} c_\Z (H , k[M]) < c_\Z (H , [M])
    \end{equation*}
    for some $k \in  \N$?
\end{question}

In {\thm} \ref{hingston positive gamma estimate}, we have answered Question \ref{Hingston question detailed} (a.k.a. Question \ref{symp ver q of hingston}) in the positive. It would be interesting to study the relation between Question \ref{Hingston question} and Question \ref{Hingston question detailed}.

\subsection{$C^0$-continuity of the $\Z$-{\specnorm}}\label{C0 conti of gamma}

Over a field {\coeff} $\K$, the {\specnorm} $\gamma_\K $ for $\C P^n$ is known to be continuous {\wrt} the $C^0$-metric on $\Ham (\C P^n)$. There are currently two known proofs for this \cite{Sh22A,Kaw22B}, but both proofs rely on the estimate \eqref{C CPn}, which does not hold for the {\specnorm} over $\Z$, c.f. {\thm} \ref{main CPn}. Therefore, $C^0$-continuity of the $\Z$-{\specnorm} for $ \C P^n$ is currently not verified apart from the $n=1$-case; the $n=1$-case was proven by Seyfaddini \cite{[Sey13]} by using two-dimensional fragmentation techniques and basic properties of {\specinv}s and the argument applies to the $\Z$-{\coeff} case. For $n>1$, we only know some partial results for $C^0$-continuity, e.g. \cite[{\thm}s 1, 4(1)]{Kaw22B} (see also \cite[Remark 25]{Kaw22B}). 

\begin{question}
    Is the {\specnorm} over $\Z$
    \begin{equation*}
        \gamma_\Z \colon \Ham (\C P^n) \to \R
    \end{equation*}
   continuous {\wrt} the  $C^0$-metric when $n>1$?
\end{question}

\subsection{Case of $S^2 \times S^2$}

 In this section, we refer to a similar phenomenon that happens on {\symp} {\mfd}s other than $\C P^n$.

 Consider $S^2 \times S^2$ equipped with the same area form on both sides so that it is a monotone {\symp} {\mfd}. There are two disjoint {\lag}s in $S^2 \times S^2$, namely the anti-diagonal $\Delta \simeq S^2$ and the Chekanov torus $T^2 _{\text{Chek}}$, and they both have non-zero {\HF} over $\mathbb{F}_3$, a field of {\charac} $3$ \cite{Smi17}. By using this, we can show that $\sup \gamma_{\F_3} = +\infty.$ The same holds over $\C$-{\coeff}s, i.e. $\sup \gamma_{\C} = +\infty.$

\subsection{$\R P^n$ over $\Z$-{\coeff}s}
Although we did not explicitly mention in the paper, all the results in Section \ref{PD general} hold also for {\lag} {\specinv}s for a monotone {\lag}.

Take $\R P^n$ inside $\C P^n$ for $n \geq 2$. We assume that $n$ is odd, so that $\R P^n$ is oriented. According to Zapolsky's computation \cite[Section 8.1]{Zap}, the {\lag} (or quantum) Floer homology over $\Z$-{\coeff}s is 
\begin{equation*}
    QH_j (\R P^n;\Z) = \begin{cases}
    \Z /2  \text{   when $j$ is odd.}\\
    0  \text{   when $j$ is even.}
    \end{cases}
\end{equation*}
This implies that every non-trivial class of $QH (\R P^n;\Z)$ is a torsion class. Thus, we have 
\begin{equation*}
    \beta_{\tor} (H) = +\infty 
\end{equation*}
for any {\hamil} $H$. In this situation, {\thm} \ref{corrected formula} is a void statement. 

\begin{question}\label{ques rpn}
    How can we express $\ell (\overline{H},[\R P^n] )$ by using information of the filtered Floer homology $HF^\tau (H;\Z)$ of $H$?
\end{question}

Note that when we work over $\Z/2$ instead of $\Z$, then we have
\begin{equation*}
    QH_j (\R P^n;\Z/2) = \Z/2 \text{      for all $j$}
\end{equation*}
and thus, we have 

\begin{equation}\label{pd rpn z2}
    \ell (\overline{H},[\R P^n] ) = - \ell (H,[pt] ) . 
\end{equation}

It seems very interesting to study the relation between Question \ref{ques rpn} and Equation \eqref{pd rpn z2}.

\subsection{Floer barcodes over the integers}

The introduction of the barcode theory from Topological Data Analysis (TDA) to symplectic topology by Polterovich--Shelukhin \cite{[PS16]} has found numerous applications in recent years. Barcode theory associates to certain $\R$-families of vector spaces, called persistence modules, objects called \textit{barcodes}. Polterovich--Shelukhin \cite{[PS16]} used this to define a \textit{Floer barcode} from the filtered Floer homology $\{HF^{<t} (H;\K)\}_{t \in \R}$, which is an $\R$-family of $\K$-vector spaces. When we work over $\Z$-{\coeff}s like in this paper, there arises an issue as the filtered Floer homology $\{HF^{<t} (H;\Z)\}_{t \in \R}$ defines an $\R$-family of $\Z$-modules instead of vector spaces. In TDA, there has been some work to define barcodes for $\R$-families of $\Z$-modules, c.f. \cite{LHP,Pat18}. It would be interesting to incorporate this progress into quantitative {\symp} topology.

\subsection{Towards quantitative Floer homotopy theory}

In recent years, some interesting results were obtained by using new {\coeff}s in Floer theory, which stem from the \textit{Floer homotopy theory}, e.g. \cite{AB, BX, Lar}. One might expect to seek quantitative applications of these Floer theories with new {\coeff}s and obtain new phenomena as in this paper. We mention the recent work of Barraud--Damian--Humili\`ere--Oancea \cite{BDHO} which aligns with this viewpoint.

\subsection{Choice of characteristic}

Finally, we list some works where the choice of the characteristic of the field to set up Floer theory plays an important role; \cite{Ton18,TV18,EL19, LN23}.

\appendix

\Addresses

\end{document}